\documentclass{article}

\usepackage[margin = 0.9in]{geometry}

\usepackage{titlesec} 

\newcommand{\appendixsection}[1]{%
  \refstepcounter{section}%
  \section*{Appendix \Alph{section}. #1}%
  \addcontentsline{toc}{section}{Appendix \Alph{section}. #1}%
}

\usepackage{tikz}
\usepackage{tikz-cd}
\usetikzlibrary{cd, arrows.meta, calc}
\usepackage{url, hyperref}
\usepackage{float}

\tikzset{->-/.style={decoration={
  markings,
  mark=at position .45 with {\arrow{>}}},postaction={decorate}}}
\usepackage{amsmath}
\usepackage{stmaryrd}
\usepackage{amssymb}
\usepackage{enumitem}
\usepackage{graphicx}
\usepackage{mathdots}
\usepackage{color}
\usepackage{diagbox}
\usepackage{array, makecell}
\usepackage{rotating}
\usepackage{amsthm}
\usepackage{dsfont}
\usepackage{adjustbox}
\usetikzlibrary{arrows}
\usepackage{mathtools}

\newcommand{\rigid}{\mathbin{%
  \vcenter{\hbox{%
    \tikz[line cap=round, line join=round, baseline={(current bounding box.center)}]{
      \def\h{0.9ex}    
      \def\dx{0.6ex}   
      \def\gap{0.5ex}  
      \def\lw{0.5pt}   

      \draw[line width=\lw] (0,-\h) -- (\dx,\h);
      \draw[line width=\lw] (\gap,-\h) -- ({\gap+\dx},\h);

      \draw[line width=\lw] (\dx,\h) -- ({\gap+\dx},\h);
      \draw[line width=\lw] (0,-\h) -- (\gap,-\h);
    }
  }}%
}}

\theoremstyle{definition}
\newtheorem{definition}{Definition}[section]

\newtheorem{remark}[definition]{Remark}
\newtheorem{notation}[definition]{Notation}
\newtheorem{setup}[definition]{Setup}

\theoremstyle{theorem}
\newtheorem{theorem}[definition]{Theorem}

\newtheorem{proposition}[definition]{Proposition}
\newtheorem{corollary}[definition]{Corollary}
\newtheorem{lemma}[definition]{Lemma}

\def\Z{\mathbb{Z}}

\def\R{\mathbb{R}}

\def\P{\mathbb{P}}

\def\R{\mathcal{R}}
\def\CH{\mathrm{CH}}
\def\Sym{\mathrm{Sym}}
\def\RH{\mathcal{RH}}
\def\Sym{\mathrm{Sym}}
\def\det{\mathrm{det}}
\def\Gm{\mathbb{G}_m}
\def\GL{\mathrm{GL}}
\def\GI{(\Gm \times \Gm) \rtimes \mu_2}

\def\PGL{\mathrm{PGL}}

\def\cO{\mathcal{O}}
\def\uD{\underline{\Delta}}
\def\Pic{\mathrm{Pic}}
\def\cH{\mathcal{H}}
\def\cD{\mathcal{D}}
\def\cX{\mathcal{X}}
\def\cY{\mathcal{Y}}
\def\cZ{\mathcal{Z}}
\def\cL{\mathcal{L}}

\def\cM{\mathcal{M}}
\def\cP{\mathcal{P}}

\def\wRH{\widetilde{\mathcal{RH}}}
\def\Aut{\mathrm{Aut}}
\def\cN{\mathcal{N}}
\def\cE{\mathcal{E}}

\title{The Integral Chow Rings of the Moduli Stacks of Hyperelliptic Prym Pairs III}
\author{Alessio Cela and Alberto Landi}
\date{\vspace{-5ex}}

\begin{document}

\maketitle

\begin{abstract}
This paper is the third and final part of a series devoted to the description of the integral Chow rings of the moduli stacks of hyperelliptic Prym pairs. For a fixed genus $g$, there are two natural stacks, $\RH_g$ and $\wRH_g$, parametrizing hyperelliptic Prym pairs, with the former being the $\mu_2$-rigidification of the latter. Both decompose as the disjoint union of $\lfloor (g+1)/2 \rfloor$ components, denoted $\mathcal{RH}_g^n$ and $\wRH_g^n$ for $n = 1, \ldots, \lfloor (g+1)/2 \rfloor$. In this paper we present quotient stack descriptions of the components $\mathcal{RH}_g^n$ for even $g$ and compute their integral Chow rings, thereby completing the computation for all irreducible components of $\RH_g$. In addition, we give quotient stack presentations for all irreducible components of $\widetilde{\mathcal{RH}}_g$ and determine when the rigidification map $\widetilde{\mathcal{RH}}_g^n \to \RH_g^n$ is a root gerbe.  We then use this to compute the Chow rings of $\wRH_g^n$ for all $g$ and $n$, with the sole exception of the case where $g$ is odd and $n=(g+1)/2$.

Finally, in the appendix, we discuss $G$-gerbes induced by an homomorphism of abelian groups $H \to G$ and an $H$-gerbe.
\end{abstract}

\tableofcontents

\section{Introduction}

A genus $g$ Prym pair is a pair $(C,\eta)$, where $C$ is a smooth, connected curve of genus $g$ 
and $\eta$ is a $2$-torsion line bundle on $C$. 
Their moduli stacks have been intensively studied over the past decades  \cite{Mumford, Farkas, Bea, Beabis, DonagiI} due to their several interactions  with other moduli stacks of interest, namely curves, abelian varieties, line bundles on curves and double covers of curves. 
More precisely, we have a commutative diagram

\begin{equation}\label{eqn: diagram intro}
\begin{tikzcd}[row sep=small, column sep=large]
 & \widetilde{\R}_g = \mathsf{Hurw}_{g' \xrightarrow{2:1} g} 
     \arrow[dl, "\mathsf{Prym}" above] \arrow[d, "r_g"] \arrow[rr] & & \widetilde{\mathcal{J}}^0_{g} \arrow[d] \\
 \mathcal{A}_{g-1} & \R_g := \widetilde{\R}_g\! \rigid\!\!\! \mu_2 \arrow[dr] \arrow[rr] & & \mathcal{J}^0_g=  \widetilde{\mathcal{J}}^0_{g}\! \rigid\!\!\! \Gm\arrow[dl] \\
 & & \cM_g &
\end{tikzcd}
\end{equation}
where 
\begin{enumerate}
    \item[$\bullet$] $\widetilde{\R}_g$ denotes the (non-rigidified) moduli stack of Prym pairs, which is also naturally identified with the stack $\mathsf{Hurw}_{g' \xrightarrow{2:1} g}$ of unramified degree $2$ covers 
     of a genus $g$ curve by a connected curve of genus $g' = 2g - 1$.
     \item[$\bullet$] The map $r_g$ is the rigidification map, obtained by killing the automorphism of the pair $(C,\eta)$ 
          corresponding to multiplication by $-1$ on $\eta$.
    \item[$\bullet$] The map $\mathsf{Prym}$ is the Prym map to the moduli stack 
          $\mathcal{A}_{g-1}$ of principally polarized abelian varieties of dimension $g-1$.
    \item[$\bullet$] The map to $\cM_g$ simply remembers the genus $g$ curve $C$.
    \item[$\bullet$] $\widetilde{\mathcal{J}}^0_{g}$ is the universal Picard stack parametrizing degree $0$ line bundles on genus $g$ curves. This is a $\Gm$-gerbe over its rigidification $\mathcal{J}^0_g$. See also~\cite[Proposition 11]{CIL24}.
\end{enumerate}

In this paper, we examine the behavior of the stacks $\wRH_g$ and $\RH_g$, which are the restrictions of 
$\widetilde{\R}_g$ and $\R_g$ to the hyperelliptic locus. 
These stacks naturally decompose into $\lfloor (g+1)/2 \rfloor$ irreducible components, 
which we denote by $\wRH_g^n$ and $\RH_g^n$ for $n = 1, \dots, \lfloor (g+1)/2 \rfloor$. 
A first question we address is: for each $n$, when does the induced map
\[
r_g^n : \wRH_g^n \longrightarrow \RH_g^n,
\]
obtained by restricting the rigidification map $r_g$, define a root gerbe?

This article is the final part of a trilogy of papers devoted to the computation of the Chow rings of the stacks $\RH_g^n$ 
\cite{CLI, CLII}, building on the initial computation of $\R_2$ in \cite{CIL24}. 
Our focus here is on the case of even genus $g$, in which case we compute the integral Chow rings of all components $\RH_g^n$. We then turn to the non-rigidified stacks $\wRH_g^n$, providing explicit presentations and computing their integral Chow rings for all $g$ and $n$, with the sole exception of the case $g$ is odd and $n = (g+1)/2$ (where we still provide presentations).

The study of the Chow rings of moduli stacks of stable curves of a fixed genus was initiated by Mumford \cite{Mumford}. 
For rational coefficients, these rings are known up to genus $9$ \cite{Faber, Izadi, PenevVakil, SamHannah}. The rational Chow rings of the stacks of (nodal) curves embedded in $\P^2$ of a fixed degree is also known \cite{CKY}.
By contrast, their integral Chow rings remain far less understood. 
Edidin and Graham introduced a theory of intersection for global quotient stacks with integer coefficients \cite{EG98}, 
providing a more refined invariant. As expected, integral computations are considerably more challenging than their rational counterparts. 
For instance, the integral Chow group of $\mathcal{M}_2$ was determined in \cite{Vis98}; those of the Deligne–Mumford compactification $\overline{\mathcal{M}}_2$ were only very recently computed independently in \cite{Lar19, DLVis-M2}; and the Chow groups of $\overline{\mathcal{M}}_3$ have now been described after inverting $6$ \cite{Per23}.

By comparison, the integral structure of the Chow rings of Prym moduli stacks had been much less explored prior to this series of papers, although their rational Picard groups were already known~\cite{Putman}. An interesting direction from here would be to study the left-hand side of diagram~\eqref{eqn: diagram intro},  namely the map $\mathsf{Prym}$ and the induced maps on Chow rings.  We expect that our results will be useful in this context.

\subsection{Main (formal) definitions}

We begin by briefly recalling the definition of Prym curves and their associated moduli stacks.
We refer to~\cite[Section 2.1]{CIL24} and~\cite[Section 1.1]{CLI} for more details.  

A \emph{Prym curve} of genus $g$ over a scheme $S$ is a triple $(C/S, \eta, \beta)$ where  
\begin{itemize}
    \item $C \to S$ is a smooth proper curve of genus $g$ with geometrically connected fibers;  
    \item $\eta \in \mathrm{Pic}(C)$ is a line bundle that restricts to a nontrivial element on each geometric fiber of $C \to S$;  
    \item $\beta: \eta^{\otimes 2} \xrightarrow{\sim} \mathcal{O}_C$ is an isomorphism of line bundles.  
\end{itemize}

The moduli stack $\mathcal{R}_g$ is obtained as the stackification of the prestack whose objects over $S$ are genus $g$ Prym curves $(C \to S, \eta, \beta)$.  
Morphisms will play a central role in this paper: a morphism $(C \to S, \eta, \beta)$ to $(C' \to S', \eta', \beta')$ consists of a cartesian diagram
\begin{equation}\label{eqn: morphism prestack alt}
    \begin{tikzcd}
    C \arrow{r}{\varphi} \arrow{d} & C' \arrow{d} \\
    S \arrow{r}{f} & S'
    \end{tikzcd}
\end{equation}
such that there exists an isomorphism $\tau: \varphi^* \eta' \to \eta$ for which the diagram
\begin{equation}\label{eqn: diagram of sheaves alt}
    \begin{tikzcd}
    \varphi^*(\eta'^{\otimes 2}) \arrow{r}{\tau^{\otimes 2}} \arrow{d}[swap]{\varphi^*(\beta')} & \eta^{\otimes 2} \arrow{d}{\beta} \\
    \varphi^* \mathcal{O}_{C'} \arrow{r} & \mathcal{O}_C
    \end{tikzcd}
\end{equation}
commutes.  

The isomorphism $\tau$ is not part of the data in $\mathcal{R}_g$; keeping track of it instead yields a $\mu_2$-gerbe over $\mathcal{R}_g$, denoted $\widetilde{\mathcal{R}}_g$ in \cite{CIL24}.  

The stacks of \emph{hyperelliptic Prym pairs}, denoted $\mathcal{RH}_g$ and 
$\widetilde{\mathcal{RH}}_g$, are defined as the fiber product of the hyperelliptic 
locus $\mathcal{H}_g \subseteq \mathcal{M}_g$ with $\mathcal{R}_g$ and 
$\widetilde{\mathcal{R}}_g$, respectively, along the natural forgetful morphisms to 
$\mathcal{M}_g$.  The stack $\RH_g$ (and similarly $\wRH_g$) decomposes into open and closed substacks
\begin{equation}\label{eqn: decomposition RH alt2}
\mathcal{R}\mathcal{H}_g = \bigsqcup_{1 \leq n \leq \lfloor\tfrac{g+1}{2}\rfloor} \mathcal{R}\mathcal{H}_g^n,
\end{equation}
which can be viewed as the stack-theoretic analogue of the set-theoretic partition for a genus $g$ curve $C$:
\[
    \mathrm{Pic}^0(C)[2] \setminus \{ \mathcal{O}_C \} 
    = \bigsqcup_{1 \leq n \leq \lfloor\tfrac{g+1}{2}\rfloor} B_n,
\]
where $B_n \subseteq \mathrm{Pic}^0(C)[2] \setminus \{ \mathcal{O}_C \}$ denotes the subset of line bundles expressible as $(n \cdot g^1_2) \otimes \cO_C(-e)$, with $e$ an effective divisor of degree $2n$ consisting of $2n$ distinct Weierstrass points. Here $g^1_2$ denotes the unique $g^1_2$ on $C$.

\subsection{Statement of the main results}

\subsubsection{The Chow ring of $\RH_g^n$ for $g$ even}

Let $g \geq 2$ be an even genus. Our first result is a presentation for the stacks $\RH_g^n$ for all $n=1,\ldots,g/2$. To state them we require some notation.

\begin{notation}
    We denote by $\chi$ the standard $\Gm$ representation and by $V$ the standard $\GL_2$ representation. Furthermore, we let $ t=c_1(\chi)$ be the first Chern class of $\chi$ in $\CH^*(B\Gm)$ and $c_i= c_i(V)$ be the Chern classes of $V$ in $\CH^*(B \GL_2)$. We will also use the same notation for their pullbacks to other stacks.

    For each $j \geq 1$, we denote by $W_j$ the $(2j+1)$-dimensional representation of $\PGL_2$ whose underlying vector space is the space of homogeneous polynomials of degree $2j$ in two variables. The action is given by $
    [B] \cdot f(X,Y) = \det(B)^j \, (f \circ B^{-1})(X,Y)$ for $[B] \in \PGL_2$ and $f$ a homogeneous polynomial of degree $2j$.
\end{notation}

\begin{theorem}\label{thm: pres RH}
   Suppose $g \geq 2$ is even. Then for every $n=1,\ldots,g/2$, we have an isomorphism of stacks
   $$
   \RH_g^n \cong \Bigg[  \frac{ ( \chi^{-1} \otimes \det(V)^{\otimes n} \otimes \Sym^{2n}(V^\vee) )   \times ( \chi \otimes \det(V)^{g-n} \otimes \Sym^{2g+2-2n}(V^\vee)) \smallsetminus \Delta}{ \Gm \times \GL_2}  \Bigg]
   $$
   where $\Delta $ is the locus of pairs of polynomials $(f,g)$ such that $fg$ is singular.
\end{theorem}

The proof is presented in~\S\ref{sec: pres for g even}. In particular, we can express $\RH_g^n$ is a $\Gm^{\times 2}$-torsor 
$$
\RH_g^n \to \bigg[  \frac{\P(\Sym^{2n}(V^\vee)) \times \P(\Sym^{2g+2-2n}(V^\vee)) \smallsetminus \uD }{\Gm \times \GL_2} \bigg]
$$
where $\uD$ denotes again the locus of pairs of polynomials $(f,g)$ (up to scalar) with $fg$ singular. 

Together with \cite[Proposition~3.5]{CLII}, this reduces the computation of $\CH^*(\RH_g^n)$ to that of $\CH^*_{\GL_2}(\P(\Sym^{2n}(V^\vee)) \times \P(\Sym^{2g+2-2n}(V^\vee)) \smallsetminus \uD )$.

\begin{theorem}\label{thm: chow P x P}
    For $a,b \geq 1$, and if the characteristic of the base field is 0 or greater than $\max(2a,2b)$, we have
    $$
    \CH^*_{\GL_2}(({\P(\Sym^{2a}(V^\vee)) \times \P(\Sym^{2b}(V^\vee))\setminus\uD})=\frac{\Z[\xi_{2a},\xi_{2b},c_1,c_2]}{J}
    $$
    where $\xi_{2j}=c_1^{\GL_2}(\cO_{\P(\Sym^{2j}(V^{\vee})}(1))$ and $J$ is the ideal generated by the following relations:
    \begin{itemize}
        \item $2(2a-1)\xi_{2a}-2a(2a-1)c_1$,
        \item $\xi_{2a}^2-c_1\xi_{2a}-4a(a-1)c_2$,
        \item $2(2b-1)\xi_{2b}-2b(2b-1)c_1$,
        \item $\xi_{2b}^2-c_1\xi_{2b}-4b(b-1)c_2$,
        \item $2b\xi_{2a}+2a\xi_{2b}-4abc_1$,
        \item $\xi_{2a}\xi_{2b}-4abc_2$,
        \item $2ab(2a-1)(2b-1)(4c_2-c_1^2)$.
    \end{itemize}
\end{theorem}

The proof is presented in~\S\ref{sec: Chow RHg even}. As explained above, from this one immediately obtains:

\begin{theorem}\label{thm: Chow RH g even}
    Suppose $g \geq 2$ is even. Then for every $n=1,\ldots,g/2$, the integral Chow ring of $\RH_g^n$ is given by
    $$
    \CH^*(\RH_g^n)= \frac{\Z[t,c_1,c_2]}{I}
    $$
    where $I$ is the ideal generated by the following relations:
    \begin{itemize}
        \item $2(2n-1)t$,
        \item $t^2-(2n-1)c_1t+n(n-1)c_1^2-4n(n-1)c_2$,
        \item $4gt-2(2g+1-2n)c_1$,
        \item $t^2+(2g-1-2n)c_1t+(g-n)(g-n-1)c_1^2-4(g+1-n)(g-n)c_2$,
        \item $2gt+2nc_1$,
        \item $t^2-(2n-g)c_1t-n(g-n)c_1^2+4n(g+1-n)c_2$,
        \item $2n(2n-1)(g+1-n)(2g+1-2n)(4c_2-c_1^2)$.
    \end{itemize}
\end{theorem}

Our proof of the previous theorem applies for $n=1$ as well, thus giving an alternative proof of~\cite[Theorem 1.10]{CLI}, after the substitutions $\beta_i=c_i$ and $\gamma=t$. Specializing instead to $g=2$, we also recover~\cite[Theorem 4]{CIL24}, after the substitution $\lambda_i= (-1)^i c_i$ and $\gamma=t$.

In section \S\ref{sec: interpretation generators g even} we interpret the above generators $t,c_1,c_2$ as Chern classes of natural vector bundles on the stack $\RH_g^n$. 

\subsubsection{Non-rigidified Prym Pairs}

Fix $g \geq 2$ and $n \in \{1,\ldots,(g+1)/2\}$. In this section we provide presentations of the non-rigidified stacks $\wRH_g^n$ and determine when the rigidification map 
\begin{equation}\label{eqn: rigidification map}
    r_g^n \colon \wRH_g^n \longrightarrow \RH_g^n
\end{equation}
is a $\mu_2$-root gerbe.
Our main geometric input is the next proposition, whose statement requires some preliminary facts and notation.

Recall that $\cH_g$ denotes the moduli stack of hyperelliptic curves of genus $g$. For $g \geq 2$, this coincides with the stack parametrizing double covers $C \to \P^1$ by connected genus $g$ curves. This last description also makes sense when $g=0$ or $g=1$, and in these cases we still write $\cH_0$ and $\cH_1$. These stacks already appeared in~\cite{AV04}, where they are denoted by $\cH_{\mathrm{sm}}(1,2,g+1)$.

For $a \geq 1$, let $\mathcal{D}_a$ be the moduli stack parametrizing pairs 
$
(\mathcal{P} \to S, D_a \subseteq \mathcal{P}),
$
where $\mathcal{P} \to S$ is a Brauer–Severi scheme of relative dimension $1$, and $D_a$ is a Cartier divisor on $\mathcal{P}$ that is finite \'etale of degree $a$ over $S$. Note that $\cD_{2a} \subset [\P(W_{a})/\PGL_2]$ is the open substack obtained by removing the locus of singular polynomials. There is a natural morphism 
$
\mathcal{H}_g \longrightarrow \mathcal{D}_{2g+2},
$
which sends a genus-$g$ hyperelliptic curve $C$ to the branch divisor of its unique (up to the $\mathrm{PGL}_2$-action) double cover $C \to \mathbb{P}^1$.

For $a,b \geq 1$, define $\mathcal{D}_{a,b}$ as the open substack of 
$
\mathcal{D}_a \times_{B\mathrm{PGL}_2} \mathcal{D}_b
$
parametrizing triples $(\mathcal{P} \to S, D_a, D_b)$, where $D_a$ and $D_b$ are disjoint Cartier divisors on $\mathcal{P}$, finite \'etale of degrees $a$ and $b$, respectively, over $S$. There is a canonical map $\cD_{a,b} \to \cD_{a+b}$ taking the sum of the divisors.

\begin{proposition}\label{prop: starting point}
For every $g \geq 2$ and $n \in \{1,\ldots,(g+1)/2\}$, there is an isomorphism of stacks:
\begin{itemize}
    \item if $n < (g+1)/2$, then
    \[
    \wRH_g^n \;\cong\; (\cH_{n-1} \times_{B \PGL_2} \cH_{g-n}) \setminus \Delta,
    \]
    \item if $g$ is odd and $n = (g+1)/2$, then
    \[
    \wRH_g^n \;\cong\; \left[\frac{(\cH_{\frac{g-1}{2}} \times_{B \PGL_2} \cH_{\frac{g-1}{2}}) \setminus \Delta}{\mu_2}\right],
    \]
    where $\mu_2$ exchanges the two factors $\cH_{\frac{g-1}{2}}$.
\end{itemize}
Here the fiber product is taken with respect to the natural maps to $B \PGL_2$, and $\Delta$ denotes the preimage of the locus of singular polynomials under the composition
\[
\cH_{n-1} \times_{B \PGL_2} \cH_{g-n}
   \;\longrightarrow\; \cD_{2n,\,2g+2-2n}
   \;\longrightarrow\; \cD_{2g+2}
   \;\hookrightarrow\; [\P(W_{g+1})/\PGL_2].
\]
\end{proposition}

The proof is given in~\S\ref{sec: geometric input} and relies on identifying the stack $\wRH_g$ with the stack of unramified degree-$2$ Hurwitz covers $C' \to C$ of a genus-$g$ curve. In particular, the above proposition provides a different geometric interpretation of the discrete invariant $n$ appearing in the decomposition \eqref{eqn: decomposition RH alt2}: the curve $C$ is obtained as a fiber product over $\P^1$ of two degree-$2$ covers of $\P^1$, one ramified at $n$ points and the other at $g+1-n$ points.  This is parallel to Beauville’s statement~\cite{Bea89}[Example (a) page 609] that the Prym variety associated with $C$ is simply the product of the Jacobians of the two double covers above.

Quotient stack resentations of the stack $\cH_g$ are known for all $g$ (\cite[Corollary 4.7]{AV04}), and we use those and the above proposition to derive presentations of all $\wRH_g^n$.

\begin{theorem}\label{thm: g even presentation tilde}
    Suppose $g \geq 2$ is even and $n \in \{1,\ldots,g/2\}$. Then,
    \begin{itemize}
        \item[(i)] if $n$ is even, we have an isomorphism of stacks
    $$
    \wRH_g^n \cong \bigg[  \frac{(\chi^{-2}\otimes\Sym^{2n}(V^{\vee})\otimes\det(V)^{\otimes n}\times\Sym^{2g+2-2n}(V^{\vee})\otimes\det(V)^{\otimes g-n}))\setminus\Delta }{\Gm\times\GL_2} \bigg];
    $$
    where $\Delta$ is the locus of polynomials $(f,g)$ such that $fg$ is singular.
    \item[(ii)] if $n$ is odd, we have
    $$
    \wRH_g^n \cong \bigg[  \frac{(\Sym^{2n}(V^{\vee})\otimes\det(V)^{\otimes n-1}\times\chi^{-2}\otimes\Sym^{2g+2-2n}(V^{\vee})\otimes\det(V)^{\otimes g+1-n})\setminus\Delta}{\Gm\times\GL_2} \bigg].
    $$
    where $\Delta$ is the locus of polynomials $(f,g)$ such that $fg$ is singular.
    \end{itemize}
\end{theorem}

\begin{theorem}\label{thm: pres wRHg for g odd}
    Suppose $g \geq 2$  is odd and $n \in \{1,\ldots, (g-1)/2\}$. Then,
    \begin{itemize}
        \item[(i)] if $n$ is even,we have an isomorphism of stacks
    $$
    \wRH_g^n \cong \bigg[  \frac{(\chi^{(1)})^{-2} \otimes W_n \times (\chi^{(2)})^{-2} \otimes W_{g+1-n} \smallsetminus \Delta}{ \Gm^{(1)} \times \Gm^{(2)} \times \PGL_2} \bigg]
    $$
    where $\Delta$ is the locus of polynomials $(f,g)$ such that $fg$ is singular. The groups $\Gm^{(1)}$ and $\Gm^{(2)}$ are both copies of $\Gm$, with the superscript indicating that $\Gm^{(i)}$ acts only on the $i$-th factor of the product. Similarly, $\chi^{(i)}$ represents the character $\chi$, with the superscript specifying which $\mathbb{G}_m$ acts on it;
    \item[(ii)] if $n$ is odd, we have 
    $$
    \wRH_g^n \cong \bigg[  \frac{\Sym^{2n}(V^\vee) \otimes \det(V)^{\otimes n-1} \times \chi^{\otimes -2} \otimes \Sym^{2g+2-2n}(V^\vee) \otimes \det(V)^{\otimes g-n} \smallsetminus \Delta}{\Gm \times \GL_2} \bigg]
    $$
    where $\Delta$ is the locus of polynomials $(f,g)$ such that $fg$ is singular.
    \end{itemize}
\end{theorem}

As is clear from Proposition~\ref{prop: starting point}, the case $g$ odd and $n=(g+1)/2$ is special.

\begin{notation}\label{not: representation A of G}
    Let $G:=\GI \subseteq \GL_2$ be the subgroup of matrices preserving the set of lines $\{ k (1,0), k(0,1)\} \subseteq k^2$. Equivalently, this is the subgroup of $\GL_2$ consisting of matrices of the form 
    $$
    (a,b;0):=
    \begin{pmatrix}
    a & 0 \\
    0 & b 
    \end{pmatrix}
    \ \text{or} \ 
    (a,b;1):=
    \begin{pmatrix}
    0 & a \\
    b & 0 
    \end{pmatrix}
    $$
    for $a,b \in k^*$. Here $k$ is the ground field. We will also regard $V$ as a $G$-representation via the inclusion $G \hookrightarrow \GL_2$.  
    Finally, we define $A$ to be the $G$-representation obtained by letting $G$ acting on $V^\vee$ via the group homomorphism 
    \[
        G \longrightarrow G, \qquad (a,b;\varepsilon) \longmapsto (a^2,b^2;\varepsilon).
    \]
\end{notation}

\begin{notation}
    Let $H: = (\Gm \times \GL_2) \rtimes \mu_2$ where the non-trivial element of $\mu_2$ acts on $(\lambda,A) \in \Gm \times \GL_2$ sending it to $(\lambda^{-1},\lambda A)$. We denote an element of $H$ by $(\lambda,A; \varepsilon)$.

    Let also $B$ be the $2$-dimensional representation of $H$ given by
    $$
    (\lambda,A; \varepsilon) \cdot (x,y):= 
    \begin{cases}
    (\det A)^{-1}\cdot(x,\lambda^{-2} y ) & \text{for } \varepsilon =0;
    \\  (\det A)^{-1}\cdot(y, \lambda^{-2} x) & \text{for } \varepsilon=1.
    \end{cases}
    $$
\end{notation}

\begin{theorem}\label{thm: pres g+1/2}
    Suppose $g \geq 2$ is odd and that $n=(g+1)/2$. Then, 
    \begin{itemize}
        \item[(i)] if $n$ is even, we have an isomorphism of stacks
        $$
    \wRH_g^n \cong \bigg[  \frac{A \otimes W_{{\frac{g+1}{2}}} \smallsetminus \Delta}{G \times \PGL_2}  \bigg]
    $$
    where $\Delta$ is the locus of polynomials $(f,g)$ such that $fg$ is singular;
    \item[(ii)] if $n$ is odd, we have
    $$
    \wRH_g^n \cong \bigg[ \frac{ \det(V)^{\otimes \frac{g+1}{2}} \otimes \Sym^{g+1}(V^{\vee}) \otimes B \smallsetminus \Delta}{ H } \bigg]
    $$
    where $\Delta$ is the locus of polynomials $(f,g)$ such that $fg$ is singular.
    \end{itemize}
\end{theorem}

From the above presentations, one can derive the next proposition. In~\S\ref{sec: Presentations and properties of the various stacks}, we will provide several proofs of this statement, each highlighting a different aspect of the stacks under consideration.

\begin{proposition}\label{prop: root stacks}
    The morphism $r_g^n$ in Equation \eqref{eqn: rigidification map} is a $\mu_2$-root gerbe if and only if one of the following holds:
\begin{enumerate}
    \item[(i)] $g$ is even and $n$ is arbitrary. In this case, with the notation of Theorem \ref{thm: Chow RH g even}, it corresponds to adding a square root of $t$ when $n$ is even, and a square root of $c_1 + t$ when $n$ is odd;
    \item[(ii)] $g$ is odd and $n$ is even with $n \neq \tfrac{g+1}{2}$. In this case, with the notation of~\cite[Theorem 1.18]{CLI}, it corresponds to adding a square root of $\xi_{2n}$ (equivalently, $\xi_{2g+2-2n}$).
\end{enumerate}
\end{proposition}

Using \cite[Proposition 3.5]{CLI}, we immediately obtain the Chow rings of $\wRH_g^n$ in the following cases: when $g$ is even and $n$ is arbitrary, and when $g$ is odd and $n$ is even with $n \neq \tfrac{g+1}{2}$. In the first case, by Theorem~\ref{thm: Chow RH g even} we get:

\begin{corollary}\label{cor: chow wRHg for g even}
    Suppose that $g \geq 2$ is even and $n \in \{1,\ldots, g/2\}$. Then
    $$
    \CH^*(\wRH_g^n) \cong \frac{\Z[u,t,c_1,c_2]}{I+(2u-\alpha)}
    $$
    where $I$ is the ideal in Theorem~\ref{thm: Chow RH g even} and $\alpha$ is equal to $t$ when $n$ is even, and equal to $t+c_1$ when $n$ is odd.
\end{corollary}

For a geometric interpretation of the generators $u,t,c_1,c_2$, see~\S\ref{sec: interpretation generators tilde g even}. In the second case, the most symmetric presentation of the Chow ring is obtained by viewing $\wRH_g^n$ as a composite of two root-gerbes over $\cD_{2n,\,2g+2-2n}$, namely via the composition $\wRH_g^n \to \RH_g^n \to \cD_{2n.2g+2-2n}$. From~\cite[Theorems 1.17, 1.18]{CLI} we obtain:

\begin{corollary}\label{cor: chow wRHg for g odd n even}
    Suppose that $g \geq 2$ is odd and $n \in \{1,\ldots,(g-1)/2\}$ is even. Then
    $$
    \CH^*(\wRH_g^n) \cong \frac{\Z[c_1,c_2,c_3,\xi_{2n},\xi_{2g+2-2n},t_{2n},t_{2g+2-2n}]}{I+(2t_{2n}-\xi_{2n},2t_{2g+2-2n}-\xi_{2g+2-2n})}
    $$
    where $I$ is the ideal in~\cite[Theorem 1.17]{CLI} with $b=g+1-n$.
\end{corollary}

Finally, when $g$ and $n$ are both odd and $n \neq (g+1)/2$ (so $r_g^n$ is not a root gerbe), Theorem~\ref{thm: pres wRHg for g odd} visibly presents $\wRH_g^n$ as a $\Gm^{\times 2}$-torsor over
\[
  \Bigg[ \frac{
    \P(\Sym^{2n}(V^\vee)) \times \P(\Sym^{2g+2-2n}(V^\vee)) \setminus \uD }{\Gm \times \GL_2}
  \Bigg],
\]
whose Chow ring is computed in Theorem~\ref{thm: chow P x P}. From this one obtain the following:

\begin{theorem}\label{thm: Chow wRHgn odd g and n}
    Suppose $g \geq 2$ and $n \in \{1, \ldots, (g-1)/2 \}$ are both odd. Then,
    $$
    \CH^*(\wRH_g^n) \cong \frac{\Z[t,c_1,c_2]}{I}
    $$
    where $I$ is the ideal generated by the following relations:
    \begin{itemize}
        \item $2(2n-1)c_1$,
        \item $(n-1)(n-2)c_1^2-4n(n-1)c_2$,
        \item $4(2g+1-2n)t+4gc_1$,
        \item $4t^2-2(2g-2)c_1t+(g-n)(g-n-1)c_1^2-4(g+1-n)(g-n)c_2$,
        \item $4nt+2(g+1)c_1$,
        \item $-2(n-1)tc_1+(n-1)(g-n)c_1^2-4n(g+1-n)c_2$,
        \item $8n(2n-1)(g+1-n)(2g+1-2n)c_2$. 
    \end{itemize}
\end{theorem}

We give an interpretation of the generators of $\CH^*(\wRH_g^n)$ for $g$ odd in~\S\ref{sec: interpretation of the generators tilde g odd}.

The computation of the Chow rings of the stacks $\wRH_{g}^{\frac{g+1}{2}}$ 
using our presentations in Theorem \ref{thm: pres g+1/2} is not carried out in this paper. Nevertheless, it would be very interesting, 
since neither the Chow ring of the classifying stack $BH$ nor the representations $A$ and $B$ 
have appeared previously in the literature.

\subsection{Conventions}
We work over an algebraically closed field $k$. For fixed $g$ and $n$ as in the definition of $\RH_g^n$ (or $\wRH_g^n$), we assume throughout that the characteristic of $k$ is either $0$ or greater than $2g+2-2n$, unless otherwise stated.

\subsection{Acknowledgements}
We thank Dan Abramovich, Samir Canning, Brendan Hassett, Aitor Iribar Lopez, Andrea Di Lorenzo, Rahul Pandharipande, Michele Pernice, Dhruv Ranganathan, Bernardo Tarini and Suichu Zhang for several related conversations. Part of this work was carried out at the University of Cambridge during the second author’s visit, and he gratefully acknowledges its warm hospitality and excellent research environment. A.C. is supported by SNF grant P500PT-222363. A.L. is supported by NSF grant 2401358.

\section{Presentation of $\RH_g^n$ for $g$ even}\label{sec: pres for g even}

In this section we prove Theorem \ref{thm: pres RH}.

\begin{proof}[Proof of Theorem \ref{thm: pres RH}]
    By \cite[Equation 4]{CLI}, \cite[Lemma 1.1]{CLII} and \cite[Theorem 3.1]{EF09} (or \cite[Corollary 4.7]{AV04}), we have a cartesian diagram
    \[
    \begin{tikzcd}[column sep=small]
    \RH_g^{n} \arrow[r]\arrow[d] & \cH_g = \bigg[ \frac{\Sym^{2g+2}(V^\vee) \otimes \det(V)^{\otimes g} \smallsetminus \Delta}{\GL_2} \bigg] \arrow[r]\arrow[d] & B \GL_2 \arrow[d]\\ 
    \mathcal{D}_{2n,2g+2-2n}  = \Bigg[ \frac{(\chi^{(1)})^{-1} \otimes W_n \times (\chi^{(2)})^{-1} \otimes W_{g+1-n} \smallsetminus \Delta}{ \Gm^{(1)} \times \Gm^{(2)} \times \PGL_2}  \Bigg] \arrow[r] & \mathcal{D}_{2g+2} = \bigg[ \frac{\chi^{\otimes -1} \otimes W_{g+1} \smallsetminus \Delta}{\mathbb{G}_m \times \mathrm{PGL}_2} \bigg] \arrow[r] & B(\Gm \times \PGL_2)
    \end{tikzcd}
\]
where both squares are cartesian, the rightmost map is induced by the  homomorphism $A \mapsto (\det(A),[A])$. 
Note that the bottom composition factors as
\[
\cD_{2n, 2g+2-2n} \to B(\Gm \times \Gm \times \PGL_2) \to B (\Gm \times \PGL_2),
\]
where the last map is induced by $(\lambda_1,\lambda_2,[B]) \mapsto (\lambda_1 \lambda_2, [B])$.  

Moreover, we have a fiber product diagram
\begin{equation}\label{diagram: square of groups}
\begin{tikzcd}[row sep=3em]
B(\Gm \times \GL_2) 
  \arrow[r, "{(\lambda,A)\mapsto A}"] 
  \arrow[d, "{(\lambda,A) \mapsto (\lambda, \det(A)\lambda^{-1},[A])}"'] 
  & B\GL_2 
    \arrow[d, "{A \mapsto (\det(A), [A])}", yshift=-0.7ex] \\
B(\Gm^2 \times \PGL_2) 
  \arrow[r, "{(\lambda_1,\lambda_2,[B]) \mapsto (\lambda_1 \lambda_2,[B])}"', yshift=-1.2ex] 
  & B(\Gm \times \PGL_2)
\end{tikzcd}
\end{equation}
The conclusion follows immediately from this.
\end{proof}

\section{The Chow ring of $\RH_g^n$ for $g$ even}\label{sec: Chow RHg even}

In this section we prove Theorems~\ref{thm: chow P x P} and~\ref{thm: Chow RH g even}. Using Theorem~\ref{thm: pres RH}, the second follows from the first and, hence we start with proving it assuming Theorem~\ref{thm: chow P x P}, whose proof is carried out in~\S\ref{subsubsec: Chow P x P}.
	
	\begin{proof}[Proof of Theorem~\ref{thm: Chow RH g even} assuming Theorem ~\ref{thm: chow P x P}]
		By Theorem~\ref{thm: pres RH}, we know that $\RH_g^n$ is a $\Gm^2$-torsor over the product of $B\Gm$ and the quotient of $({\P(\Sym^{2n}(V^\vee)) \times \P(\Sym^{2g+2-2n}(V^\vee)))\setminus\uD}$ by $\GL_2$. Recall that we denote by $t$ the first Chern class of the standard $\Gm$-character $\chi$, and that $\CH^*(\cX\times B\Gm)\cong\CH^*(\cX)\otimes_{\mathbb{Z}}\Z[t]$ as rings, for every smooth quotient stack $\cX$, see~\cite[Lemma 10]{Oe18}.
        
        Then, the torsor is the composite of two $\Gm$-torsors relative to $-\xi_{2n}-t+nc_1$ and $-\xi_{2g+2-2n}+t+(g-n)c_1$, respectively. The statement follows from applying the $\Gm$-bundle formula twice, and some simplifications using the first relation in the statement.
	\end{proof} 
	
	\subsection{Proof of Theorem~\ref{thm: chow P x P}}\label{subsubsec: Chow P x P}
	In this section we prove Theorem ~\ref{thm: chow P x P}. We start with some notation.
	
	\begin{notation}
		For every $r\geq1$, we denote by $\P^r$ the projective bundle $\P(\Sym^r(V^{\vee}))$. We denote the class of $\cO_{\P^r}(1)$ by $\xi_r$, both in the $\GL_2$ and $\Gm^2$-equivariant Chow ring, when regarding $\Gm^2$ as the maximal torus of $\GL_2$ consisting of diagonal matrices. 
	\end{notation}
	
    An application of the excision sequence together with the projection formula yields the following lemma.
	
	\begin{lemma}\label{lem: generators RH g even}
		We have
		\[
		\CH_{\GL_2}^*(\P^{2a}\times\P^{2b})\cong\frac{\Z[\xi_{2a},\xi_{2b},c_1,c_2]}{(p_{2a}(\xi_{2a}),p_{2b}(\xi_{2b}))}
		\]
		where $p_{2j}(\xi_{2j})$ are the monic polynomials of degree $2j+1$ coming from the projective bundle formula applied to the single factors.
		
		Moreover, if we denote by $J$ the image under the $\GL_2$-equivariant pushforward along the inclusion $\uD\hookrightarrow\P^{2a}\times\P^{2b}$, then
		\[
		\CH_{\GL_2}^*(\P^{2a}\times\P^{2b}\setminus\uD)\cong\frac{\Z[\xi_{2a},\xi_{2b},c_1,c_2]}{(p_{2a}(\xi_{2a}),p_{2b}(\xi_{2b}))+J}.
		\]
	\end{lemma}
	
	To compute the ideal $J$, first notice that $\uD$ has three $\GL_2$-invariant irreducible components $\uD_1$, $\uD_2$ and $\uD_{1,2}$, parametrizing pairs $(f,g)$ such that $f$ is not square free, $g$ is not square free, and that $f,g$ share a common factor, respectively. Therefore, it is enough to compute the image of the $\GL_2$-equivariant pushforward along the inclusions $j_1:\uD_1\hookrightarrow\P^{2a}\times\P^{2b}$, $j_2:\uD_2\hookrightarrow\P^{2a}\times\P^{2b}$ and $j_{1,2}:\uD_{1,2}\hookrightarrow\P^{2a}\times\P^{2b}$. To do so, we construct a Chow envelope of the three irreducible components of $\uD$.
	
	There are squaring and multiplication maps
	\[
	\begin{tikzcd}[row sep=tiny]
		F_r:\P^r\times(\P^{2a-2r}\times\P^{2b})\arrow[r] & \P^{2a}\times\P^{2b}, & (h,(f,g))\arrow[r,mapsto] & (h^2f,g),\\
		G_r:\P^r\times(\P^{2a}\times\P^{2b-2r})\arrow[r] & \P^{2a}\times\P^{2b}, & (h,(f,g))\arrow[r,mapsto] & (f,h^2g),\\
		M_r:\P^r\times(\P^{2a-r}\times\P^{2b-r})\arrow[r] & \P^{2a}\times\P^{2b}, & (h,(f,g))\arrow[r,mapsto] & (hf,hg).
	\end{tikzcd}
	\]
	that are $\GL_2$-equivariant and factor through $\uD_1$, $\uD_2$ and $\uD_{1,2}$, respectively. Denote by $F$, $G$ and $M$ the disjoint union of the maps $F_r$, $G_r$ and $M_r$, respectively.
	\begin{lemma}\label{lem: Chow envelopes n>1 GL2 case}
		The morphisms $F$, $G$ and $M$ form $\GL_2$-equivariant Chow envelopes of $\uD_1$, $\uD_2$ and $\uD_{1,2}$, respectively, and the $\GL_2$-equivariant pushforwards are surjective.
	\end{lemma}
	\begin{proof}
		The proof is the same as~\cite[Lemma 5.2]{CLI}. This uses the assumption on the characteristic of the base field.
	\end{proof}
	
	We are left with computing the images of $F_*$, $G_*$ and $M_*$. The first two are essentially already known, and we recall the result here.
	
	\begin{lemma}\label{lem: image F and G}
		The image of $F_*$, equivalently of $i_{1*}$, is the ideal generated by
		\begin{align*}
			{F_1}_*(1) &= 2(2a-1)\xi_{2a}-2a(2a-1)c_1, \qquad {F_1}_*(\xi_1) = \xi_{2a}^2-c_1\xi_{2a}-2a(2a-2)c_2.
		\end{align*}
		Similarly, the image of $G_*$, equivalently of ${i_2}_*$ is the ideal generated by
		\begin{align*}
			{G_1}_*(1) &= 2(2b-1)\xi_{2b}-2b(2b-1)c_1, \qquad {G_1}_*(\xi_1) = \xi_{2b}^2-c_1\xi_{2b}-2b(2b-2)c_2.
		\end{align*}
	\end{lemma}
	\begin{proof}
		It follows immediately from~\cite[Proposition 4.2, Lemma 4.3]{EF09}.
	\end{proof}
	
    Notice that $p_{2j}(\xi_{2j})$ is monic of degree $\geq\min(2a+1,2b+1)\geq3$ in $\xi_{2j}$, hence the classes in Lemma~\ref{lem: image F and G} lift uniquely to $\Z[\xi_{2a},\xi_{2b},c_1,c_2]$.
	
	\begin{lemma}\label{lem: superfluous polynomials}
		The polynomials $p_{2j}(\xi_{2j})$ are contained in the ideal of $\Z[\xi_{2a},\xi_{2b},c_1,c_2]$ generated by the classes in Lemma~\ref{lem: image F and G}.
	\end{lemma}
     In particular, they are superfluous in Lemma~\ref{lem: generators RH g even} (and we will sometimes think of$J$ as an ideal in $\Z[\xi_{2a},\xi_{2b},c_1,c_2]$).
	\begin{proof}
		Let $\Gm^2$ be the maximal torus in $\GL_2$ consisting of diagonal matrices, with two standard characters $t_1$, $t_2$, induced by the two projections. The Chow ring of the quotient $[\P^{2a}\times\P^{2b}/\Gm^2]$ is $\Z[\xi_{2a},\xi_{2b},t_1,t_2]/(\widetilde{p}_{2a}(\xi_{2a}),\widetilde{p}_{2b}(\xi_{2j}))$, where the two polynomials $\widetilde{p}_{2j}(\xi_{2j})$ are the $\Gm^2$-analogues of $p_{2j}(\xi_{2j})$. By~\cite[Lemma 3.3]{CLI}, it is enough to prove that $\widetilde{p}_{2a}(\xi_{2a})$ is contained in the ideal of $\Z[\xi_{2a},\xi_{2b},t_1,t_2]$ generated by
		\begin{align*}
			(F_1^{\Gm^2})_*(1) &= 2(2a-1)\xi_{2a}-2a(2a-1)(t_1+t_2), \qquad (F_1^{\Gm^2})_*(\xi_1) = \xi_{2a}^2-(t_1+t_2)\xi-2a(2a-2)t_1t_2.
		\end{align*}
		and similarly for $\widetilde{p}_{2b}(\xi_{2b})$. By~\cite[Lemma 4.3]{EF09}, the ideal above contains the product of the $\Gm^2$-equivariant classes of two hyperplanes of $\P^{2a}$. Since $\widetilde{p}_{2a}(\xi_{2a})$ is the product of the classes of the $2a+1$ coordinate hyperplanes, it is contained in the ideal above.
	\end{proof}
	
	We are left with computing the image of $M_{r*}$ for every $r$.
	
	\begin{lemma}\label{lem: generators M}
		The image of $M_*$ is generated by $M_{r*}(1)$ for $1\leq r\leq\min(2a,2b)$ and $M_{1*}(\xi_1)$.
	\end{lemma}
	\begin{proof}
		As $M_r^*(\xi_{2j})=\xi_r+\xi_{2j-r}$ for $j\in\{a,b\}$, by the projection formula we know that the image of $M_*$ is generated by $M_{r*}(\xi_r^i)$ for $1\leq r\leq\min(2a,2b)$ and $0\leq i\leq r$. By~\cite[Proposition 3.4]{CLI} or~\cite[Lemma 2.1]{FuVi}, we can work $\Gm^2$-equivariantly, where $\Gm^2$ is the maximal torus of $\GL_2$ of diagonal matrices. We use the same strategy as in~\cite[Lemma 2.10]{Lan24}, and consider diagrams
		\[
		\begin{tikzcd}
			\P^{r-1}\times\P^{2a-r}\times\P^{2b-r}\arrow[r,"\psi",hookrightarrow]\arrow[rd,"\phi_r"',hookrightarrow] & \P^r\times\P^{2a-r}\times\P^{2b-r}\arrow[r,"M_r"] & \P^{2a}\times\P^{2b}\\
			& \P^{r-1}\times\P^{2a-r+1}\times\P^{2b-r+1}\arrow[ru,"M_{r-1}"']
		\end{tikzcd}
		\]
		for every $r\geq2$, where $\psi$ is the closed immersion of the first coordinate hyperplane, thus corresponding to $(h,f,g)\mapsto(x_0h,f,g)$ and $\phi$ sends $(h,f,g)$ to $(h,x_0f,x_0g)$. The diagram is commutative and all morphisms are $\Gm^2$-equivariant. Let $h_0:=\psi_*(1)$, then $h_0\xi_r^i=\xi_r^{i+1}+\alpha\xi_r^i$ for some $\alpha$ pulled back from $\CH^*(B\Gm^2)$ and the statement follows by lexicographic induction on $(r,i)$ from the recursive formula above.
	\end{proof}
	
	\begin{notation}\label{not: candidate ideal}
		Let $J_1$ be the ideal generated by $M_{1*}(1)$, $M_{1*}(\xi_1)$, $M_{2*}(1)$ and the classes in Lemma~\ref{lem: image F and G}. 
	\end{notation}
	
	\begin{lemma}\label{lem: ideal = candidate ideal}
		We have $J=J_1$.
	\end{lemma}
	\begin{proof}
		By Lemmas~\ref{lem: image F and G},~\ref{lem: superfluous polynomials} and Lemma~\ref{lem: generators M}, it is enough to show that $M_{r*}(1)\in J_1$ for every $r\geq2$.
		
		Consider the cartesian diagram
		\begin{equation}\label{eq: diag Mr GL2 PGL2}
		\begin{tikzcd}
			\left[\frac{\P^r\times\P^{2a-r}\times\P^{2b-r}}{\GL_2}\right]\arrow[d,"\phi_r"]\arrow[r,"M_r^{\GL_2}"] & \left[\frac{\P^{2a}\times\P^{2b}}{\GL_2}\right]\arrow[d,"\phi"]\\
			\left[\frac{\P^r\times\P^{2a-r}\times\P^{2b-r}}{\PGL_2}\right]\arrow[r,"M_r^{\PGL_2}"] & \left[\frac{\P^{2a}\times\P^{2b}}{\PGL_2}\right]
		\end{tikzcd}
		\end{equation}
	    where the action of $\PGL_2$ on the projective spaces is the one induced by $\GL_2$; the superscripts on the maps indicate the underlying group whose action we are considering, and the vertical maps are induced by the natural morphism $B\GL_2\rightarrow B\PGL_2$. Then,
		\begin{equation}\label{eq: equality Mr1 GL2 PGL2}
			(M_{r}^{\GL_2})_*(1)=(M_{r}^{\GL_2})_*(\phi_r^*(1))=\phi^*(M_r^{\PGL_2})_*(1).
		\end{equation}
		By~\cite[Section 5.2.2]{CLI}, we know that $(M_r^{\PGL_2})_*(1)$ is contained in the ideal $J_1^{\PGL_2}$ generated by classes in the image of $F^{\PGL_2}_*$, $G^{\PGL_2}_*$, $M_{1*}^{\PGL_2}$ and $M^{\PGL_2}_{2*}$. Notice that this is true also when $a$ and $b$ are 1, by definition. There are diagrams for $F$ and $G$ analogous to~\eqref{diagram: square of groups}, from which it follows that $\phi^*J_1^{\PGL_2}\subseteq J_1^{\GL_2}=J_1$. Together with equation~\eqref{eq: equality Mr1 GL2 PGL2}, this concludes.
	\end{proof}
	
	\begin{lemma}\label{lem: image M1}
		We have
		\begin{align*}
			M_{1*}(1)=2b\xi_{2a}+2a\xi_{2b}-4abc_1, && M_{1*}(\xi_1)&=\xi_{2a}\xi_{2b}-4abc_2
		\end{align*}
		and
		\begin{align*}
			M_{2*}(1)&=(2a-1)(2b-1)\xi_{2a}\xi_{2b}+b(2b-1)\xi_{2a}^2+a(2a-1)\xi_{2b}^2+4ab(a+b-1)c_2\\	&-((4a-1)b(2b-1)\xi_{2a}+a(2a-1)(4b-1)\xi_{2b})c_1+2ab(2a-1)(2b-1)c_1^2.
		\end{align*}
	\end{lemma}
	\begin{proof}
		The computation of the classes in the statement can be carried out in the same fashion as in~\cite[Lemma 5.4]{CLI} and~\cite[Lemma 4.3]{EF09}. We use an alternative approach, leveraging previous $\PGL_2$-equivariant computations.
		
		By diagram~\eqref{eq: diag Mr GL2 PGL2} with $r=1,2$, and using the torsion-freeness of the $\GL_2$-equivariant Chow ring of $\P^{2a}\times \P^{2b}$, it suffices to compute the pullback $\phi^*$ of the corresponding $\PGL_2$-equivariant classes computed in~\cite[Lemma~5.4, Corollary~5.9]{CLI}. The statement then follows from the identities
    \[
    \phi^*(\xi_{2a}) = \xi_{2a} - a c_1, \qquad
    \phi^*(\xi_{2b}) = \xi_{2b} - b c_1, \qquad
    \phi^*(c_2) = 4c_2 - c_1^2,
    \]
    together with the fact that the pullback of $\tau = c_1^{\PGL_2}(\cO_{\P^1}(2))$ along $[\P^1/\GL_2] \to [\P^1/\PGL_2]$ is $2\xi_1 - c_1$. On $[\P^{2a}/\PGL_2]$ (resp. on $[\P^{2b}/\PGL_2]$), the class $\xi_{2a}$ (resp. $\xi_{2b}$) denotes the $\PGL_2$-equivariant first Chern class of $\cO(1)$. The class $c_2$ on $B \PGL_2$ is defined as $c_2=c_2(W_2)$. The three identities above then follow from the fact that $W_j$ pulls back to $\Sym^{2j}(V^\vee)\otimes \det(V)^{\otimes j}$ under the standard morphism $B\GL_2 \to B\PGL_2$.

    Finally, the pullback $x$ of $\tau$ to $[\P^1/\GL_2]$ is of the form $2\xi_1 + u c_1$ for some $u\in\Z$. Since $\tau^2 + c_2 = 0$ on $[\P^1/\PGL_2]$ by~\cite[Lemma~5.1]{FV11}, we have
    \[
    (2\xi_1 + u c_1)^2 \;=\; x^2 \;=\; -4c_2 + c_1^2
    \;=\; 4\xi_1^2 - 4 c_1 \xi_1 + c_1^2
    \;=\; (2\xi_1 - c_1)^2,
    \]
    where we used the projective bundle relation $\xi_1^2 - c_1\xi_1 + c_2 = 0$ on $[\P^1/\GL_2]$. Expanding we obtain
    \[
    4u\,\xi_1 c_1 + u^2 c_1^2 \;=\; -4\,\xi_1 c_1 + c_1^2,
    \]
    which forces $u=-1$ as the only relation in $\xi_1$ in $\CH^*_{\GL_2}(\P^1)$ has degree $2$. Thus the pullback of $\tau$ is $2\xi_1 - c_1$, as claimed.
	\end{proof}

    \begin{remark}\label{rmk: simplification M2(1)}
        The expression for $M_{2*}(1)$ can be simplified if we work modulo the other generators of $J_1$. Indeed, by using $F_{1*}(\xi_1)$, $G_{1*}(\xi_1)$, and $M_{1*}(\xi_1)$ to eliminate $\xi_{2a}^2$, $\xi_{2b}^2$, and $\xi_{2a}\xi_{2b}$, respectively, from the expression in Lemma~\ref{lem: image M1}, we get
        \[
            -(2a-1)(2b-1)\left((2b\xi_{2a}+2a\xi_{2b})c_1-2abc_1^2-8abc_2\right).
        \]
        Using the relation $2b\xi_{2a}+2a\xi_{2b}-4abc_1=0$ coming from $M_{1*}(1)$, we can simply it further to get
        \[
            2ab(2a-1)(2b-1)(4c_2-c_1^2).
        \]
        It follow easily that $M_{2*}(1)$ is a not a redundant generator of $J$.
    \end{remark}
	
	Now, the proof of Theorem~\ref{thm: chow P x P} follows easily.
	
	\begin{proof}[proof of Theorem~\ref{thm: chow P x P}]
		By Lemmas~\ref{lem: generators RH g even} and~\ref{lem: superfluous polynomials} we know the generators, and it is enough to compute the ideal $J$. This is shown to be equal to $J_1$ in Lemma~\ref{lem: ideal = candidate ideal}, which is computed in Lemmas~\ref{lem: image F and G} and~\ref{lem: image M1}. Finally, Remark~\ref{rmk: simplification M2(1)} explains the last relation in Theorem~\ref{thm: Chow RH g even}.
	\end{proof}
\subsection{Interpretation of the generators of $\CH^*(\RH_g^n)$ for $g$ even}\label{sec: interpretation generators g even}

In this subsection we provide natural vector bundles on $\RH_g^n$ whose Chern classes generate the Chow ring, when $g$ is even.
By Theorem~\ref{thm: Chow RH g even}, the Chow ring of $\RH_g^n$ is generated by the classes $c_1$, $c_2$ and $t$ obtained by pullback from $B\Gm\times\GL_2$.

Let $V_g$ be the rank 2 vector bundle on $\cH_g$ functorially defined as
\[
    \pi_*(\omega_{\pi}^{\otimes g/2}((1-g/2)W))
\]
for every family $\pi:C\rightarrow S$ of hyperelliptic curves with Weierstrass divisor $W\subset C$; see~\cite[Section 5.1]{EF09}. We denote by the same symbol the pullback of $V_g$ along the forgetful morphism $\RH_g^n\rightarrow\cH_g$.
We also denote by $\cN_n$ the line bundle on $\cD_{2n}$ that associates
\[
    \rho_*(\omega_{\rho}^{\otimes n}\otimes\cO(F_{2n}))
\]
to every family $F_{2n}\subset P\xrightarrow{\rho} S$ of degree $2n$ étale divisors on Brauer-Severi scheme of relative dimension 1 over $S$; see~\cite[Section 5.3]{CLI}. Again, we use the same symbol for the pullback of $\cN_n$ along the composite $\RH_g^n\rightarrow\cD_{2n,2g+2-2n}\xrightarrow{\mathrm{pr}_1}\cD_{2n}$.

\begin{lemma}\label{lem: interpretation generators g even}
    We have
    \begin{align*}
        c_1=-c_1(V_g), && c_2=c_2(V_g), && t=-c_1(\cN_n).
    \end{align*}
\end{lemma}
\begin{proof}
    The first two equalities follow from~\cite[Proposition 5.2]{EF09}. \footnote{In ~\cite[Proposition 5.2]{EF09}, the authors interpret $\GL_2$ as $\mathrm{Aut}(\P^1,\cO_{\P^1}(1))$, while for us it is $\mathrm{Aut}(\P^1,\cO_{\P^1}(-1))$. In particular, the standard $\GL_2$ representation of $\GL_2$ differ by a dual and this justifies the difference of signs. See also \cite[Remark 41]{CIL24} and \cite[footnote 4]{CLI}.}

    For the last identity, from~\cite[Lemma 5.20]{CLI} and with the same notation as in that lemma, we have  $c_1(\cN_n)=\xi_{2n}$ in $\CH^*(\cD_{2n,2g+2-2n})$. With respect to the presentation of $\cD_{2n,2g+2-2n}$ in~\cite[Equation 5]{CLI}, we get $\xi_{2n}=-x_1$, where $x_1$ is the first Chern class of the standard representation of the first factor $\Gm$ in $\Gm\times\Gm\times\PGL_2$. The statement then follows immediately from the commutative diagram
    \[
    \begin{tikzcd}
        \RH_g^n\arrow[r]\arrow[d] & \cD_{2n,2g+2-2n}\arrow[d]\\
        B(\Gm\times\GL_2)\arrow[r,"\phi"] & B(\Gm\times\Gm\times\PGL_2)
    \end{tikzcd}
    \]
    where $\phi$ is induced by $(\lambda,A)\mapsto(\lambda,\lambda^{-1}\det A,[A])$.
\end{proof}

\section{Presentations and properties of the stacks $\wRH_g^n$}\label{sec: Presentations and properties of the various stacks}

\subsection{Proof of the main geometric input}\label{sec: geometric input}

As explained in the introduction, the starting point for all our presentations is Proposition \ref{prop: starting point}, which exhibits $\wRH_g^n$ as an open in a fiber product of two stacks of hyperelliptic curves.

The first step to show this is recognizing the stack $\wRH_g$ as the stack of Hurwitz covers. Let $\mathsf{Hurw}_{g' \xrightarrow{2:1} g}$ denote the stack of unramified degree $2$ covers of a genus $g$ curve by a connected genus $g'$ curve. The genus $g'$ is completely determined by the Riemann-Hurwitz formula: $g' = 2g - 1$.

\begin{lemma}
    There is an isomorphism of stacks 
    $$
    \wRH_g \cong \mathsf{Hurw}_{g' \xrightarrow{2:1} g} |_{\cH_g},
    $$
    where the right-hand side denotes the restriction of $\mathsf{Hurw}_{g' \xrightarrow{2:1} g}$ to the locus of hyperelliptic genus $g$ curves.
\end{lemma}
\begin{proof}
     This is well-known and similar to \cite{Vis98} and \cite{AV04}.
\end{proof}

Next we try to understand the stack $(\cH_{n-1} \times_{B \PGL_2} \cH_{g-n}) \smallsetminus \Delta$ better. An object of this stack over a scheme $S$ is, by definition, a triple $(\alpha_1,\alpha_2,\phi)$ where:

\begin{itemize}
    \item For $i=1,2$, $\alpha_i$ is a diagram $C_i \xrightarrow{f_i} P_i \to S$, where $f_i$ is a double cover (over $S$) of a Brauer--Severi $P_i$ of relative dimension $1$ by a connected curve of genus $n-1$;
    \item $\phi$ is an isomorphism $\phi: P_1 \to P_2$ over $S$.
\end{itemize}

These satisfy the property that the branch divisors of $f_1$ and $f_2$ do not intersect on $P_1$ (after transporting the branch divisor of $f_2$ via $\phi$).

Since the triple $(\alpha_1, \alpha_2, \phi)$ is isomorphic to $(C_1 \to P_1 \to S, \phi^*C_2 \to P_1 \to S, \mathrm{Id})$, our stack is equivalent to the stack whose objects are diagrams over $S$ of the form
\[
\begin{tikzcd}
C_1 \arrow[rd,"f_1"'] & & C_2 \arrow[ld,"f_2"] \\
 & P &
\end{tikzcd}
\]
where $C_i \to P$ are double covers of a Brauer--Severi $P \to S$ of relative dimension $1$ with $C_1$ (resp. $C_2$) of genus $n-1$ (resp. $g-n$).

Denote by $\iota_i$ the involution of the curve $C_i$ for $i=1,2$. We can complete the above diagram to a diagram 

\begin{equation}\label{diagram: all curves}
\begin{tikzcd}[row sep=3.5em, column sep=5em]
& C' := C_1 \times_P C_2 
    \arrow[ld, "p_1"'] 
    \arrow[rd, "p_2"] 
    \arrow[d, "\pi"'] 
    \arrow[loop above, "\scriptstyle{\iota_1 \times \iota_2, \; \iota_1, \; \iota_2}"] & \\
C_1 
    \arrow[rd, "f_1"'] 
    \arrow[loop left, "\iota_1"] & 
C := (C_1 \times_P C_2)/\langle \iota_1 \times \iota_2 \rangle
    \arrow[d, "f"] & 
C_2 
    \arrow[ld, "f_2"] 
    \arrow[loop right, "\iota_2"] \\
& P &
\end{tikzcd}
\end{equation}

Note that $\pi$ is finite and étale of degree $2$. Indeed, the map $C' \to P$ has degree $4$ and is branched over $2g+2$ points; above each branch point there are two simple ramification points on $C'$, so the total ramification contribution is $4g+4$. By Riemann--Hurwitz, $C'$ has genus $g' = 2g-1$, and applying Riemann--Hurwitz to $\pi$ shows that $\pi$ is étale.

    Define 
    $$
    \Phi: \bigsqcup_{1 \leq n \leq  (g+1)/2} (\cH_{n-1} \times_{B \PGL_2} \cH_{g-n}) \smallsetminus \Delta \to \mathsf{Hurw}_{g' \xrightarrow{2:1} g}|_{\cH_g}
    $$
    by associating to $(\alpha_1,\alpha_2,\phi)$  the double $C' \to C$ (over $S$) above. 

    \begin{lemma}
        The morphism $\Phi$ maps $(\cH_{n-1} \times_{B \PGL_2} \cH_{g-n})\setminus\Delta$ to $\wRH_g^n \subseteq \mathsf{Hurw}_{g' \xrightarrow{2:1} g}|_{\cH_g}$.
    \end{lemma}
    \begin{proof}
    It is not restrictive to work over closed points. Since $P\cong \P^1$, we can write
    \[
    \pi_*\cO_{C'} \;=\; \cO_C \oplus \eta,
    \quad\text{with}\quad
    \eta \;=\; f^*\cO(k)(-e_{2k})
    \]
    for a unique $k\in\{1,\ldots,\lfloor (g+1)/2\rfloor\}$, where    $e_{2k}$ is an effective divisor of $2k$ distinct Weierstrass points. Then
    \[
    f_*\pi_*\cO_{C'} \;=\; f_*(\cO_C \oplus \eta)
    \;=\; f_*\cO_C \oplus f_*\eta
    \;=\; (\cO \oplus \cO(-g-1)) \oplus \cO(k) \otimes f_*\cO_C(-e_{2k}).
    \]
    To compute the last term, push forward along $f$ (which is finite, hence affine) the exact sequence
    \[
    0 \to \cO_C(-e_{2k}) \to \cO_C \to \cO_{e_{2k}} \to 0,
    \]
    obtaining
    \[
    0 \to f_*\cO_C(-e_{2k}) \to f_*\cO_C \cong \cO\oplus\cO(-g-1)
   \to f_*\cO_{e_{2k}} \cong \cO_{f(e_{2k})} \to 0.
    \]
    A local computation shows that the map $\cO(-g-1)\to \cO_{f(e_{2k})}$ is zero, hence 
    \[
    f_*\cO_C(-e_{2k})  \cong \cO\big(-f(e_{2k})\big)\oplus \cO(-g-1)
    \cong \cO(-2k)\oplus \cO(-g-1).
    \]
    On the other hand, since Diagram~\eqref{diagram: all curves} is cartesian of flat and proper maps, we have
    \begin{align*}
    \pi_* f_* \cO_{C'} &= f_{1*}\cO_{C_1}\otimes f_{2*}\cO_{C_2} \\
    &= (\cO\oplus \cO(-n)) \otimes (\cO\oplus \cO(-g-1+n)) \\
    &= \cO \oplus \cO(-n) \oplus \cO(-g-1+n) \oplus \cO(-g-1).
    \end{align*}
Comparing the two decompositions shows that $k=n$.
    \end{proof}

    Let 
    $$
    \Phi_g^n: (\cH_{n-1} \times_{B \PGL_2} \cH_{g-n}) \smallsetminus \Delta \to \wRH_g^n
    $$
    be the induced map.

\begin{proof}[Proof of Proposition \ref{prop: starting point}]

    We distinguish two cases:
    \begin{itemize}
        \item Suppose $n \neq (g+1)/2$.

        In this case, we will prove that $\Phi_g^n$ is an isomorphism by explicitly giving its inverse. Given an unramified double cover $\pi: C' \to C$ of an hyperelliptic curve $C$, we have associated to it a double cover $f: C \to P$ where $P \to S$ is a Brauer--Severi of relative dimension $1$ (for instance, one can take 
    $P$ to be the image of the canonical map). Let $\iota$ be the involution of $C$ preserving $f$ and $\kappa$ the involution on $C'$ preserving $\pi$. We can lift $\iota$ to an automorphism $\tilde{\iota}$ of $C'$ such that $\pi \circ \iota = \tilde{\iota} \circ \pi$. Indeed, $\iota$ acts as multiplication by $-1$ on the relative Jacobian of $C/S$, and thus preserves every $2$-torsion line bundle. In particular, since 
    $
    C' = \operatorname{Spec}_C\!\left(\mathcal{O}_C \oplus \eta\right)
    $
    for some $2$-torsion line bundle $\eta$, the action of $\iota$ on $\eta$ induces an automorphism of $C'$, which is precisely 
    our desired lift $\tilde{\iota}$.
    
    The subgroup 
    $$
    \langle \kappa, \tilde{\iota} \rangle \subseteq \mathrm{Aut}(C'/P)
    $$ 
    is isomorphic to $\mu_2^{\times 2}$. Taking the quotient of $C'$ by the three non-trivial elements in $\langle \kappa, \tilde{\iota} \rangle$ yields three quotient maps corresponding to $p_1,p_2$ and $\pi$ in diagram \eqref{diagram: all curves}. This defines the inverse of $\Phi_g^n$. To check that the resulting diagram is cartesian, it suffices to note that $C'$ admits a natural map to the fiber product, and that both $C' \to C_1$ and $C_1 \times_{P} C_2 \to C_1$ have degree $2$. Hence the map 
    $
    C' \longrightarrow C_1 \times_{P} C_2
    $
    must be an isomorphism;
    \item Suppose $n = (g+1)/2$.  

    In this case, $n-1$ and $g-n$ are equal, and the morphism $\Phi^n_g$ factors through 
    $
    \bigl[(\cH_{n-1} \times_{B \PGL_2} \cH_{g-n}) \setminus \Delta \big/ \mu_2\bigr],
    $
    as $n-1$ and $g-n$ are indistinguishable. For the same reason, the map described above is only defined into this quotient. Apart from this, the proof proceeds in the same way.
    \end{itemize}
\end{proof}

\subsection{Case $g$ even}\label{sec: case g even}

In this section we prove Theorem \ref{thm: g even presentation tilde} and $(i)$ in Proposition \ref{prop: root stacks}. We start with the presentations.

\begin{proof}[Proof of Theorem \ref{thm: g even presentation tilde}]
    
Suppose that $n$ is even, if $n$ is odd then the argument is the same. By Proposition \ref{prop: starting point}, \cite[Theorem 3.1]{EF09} and \cite[Corollary 4.7]{AV04}, we have a commutative diagram
\begin{equation}\label{eq: diag pres wRh}
\begin{tikzcd}[row sep=3em]
\wRH_g^n 
  \arrow[r] 
  \arrow[d] 
  & \cH_{n-1} = \bigg[ \frac{\chi^{\otimes - 2} \otimes W_n \smallsetminus \Delta}{\Gm \times \PGL_2}  \bigg]
    \arrow[d] \\
\cH_{g-n} =\bigg[ \frac{\Sym^{2g+2-2n}(V^\vee) \otimes \det(V)^{\otimes g-n} \smallsetminus \Delta}{\GL_2} \bigg]
  \arrow[r] 
  & B\PGL_2
\end{tikzcd}
\end{equation}
realizing $\wRH_g^n$ as an open in the cartesian product. Since the right vertical arrow factors though $B(\Gm \times \PGL_2) \to B \PGL_2$ mapping $(\lambda,[B]) \mapsto [B]$ and the bottom horizontal arrow factors through $B\GL_2 \to B\PGL_2$ mapping $A \mapsto [A]$, we can expand the above diagram the above diagram to a cartesian diagram

\begin{equation}\label{eq: big diag pres wRH}
\begin{tikzcd}[row sep=2em, column sep=2.5em]
\cH_{n-1} \times_{B \PGL_2} \cH_{g-n} \arrow[r] \arrow[d] 
  & \bigg[ \frac{\chi^{\otimes -2} \otimes \Sym^{2n}(V^\vee) \otimes \det(V)^{\otimes n} \smallsetminus \Delta }{\Gm \times \GL_2} \bigg] \arrow[r] \arrow[d] 
  & \bigg[ \frac{\chi^{-\otimes 2} \otimes W_n \smallsetminus \Delta}{\Gm \times \PGL_2}  \bigg] \arrow[d] \\
\bigg[  \frac{\Sym^{2g-2n}(V^\vee) \otimes \det(V)^{\otimes g-n} \smallsetminus \Delta}{\Gm \times \GL_2} \bigg] \arrow[r] \arrow[d] 
  & B(\Gm \times \GL_2) \arrow[r] \arrow[d] 
  & B(\Gm \times \PGL_2) \arrow[d] \\
\bigg[ \frac{\Sym^{2g-2n}(V^\vee) \otimes \det(V)^{\otimes g-n} \smallsetminus \Delta}{\GL_2} \bigg] \arrow[r] 
  & B\GL_2 \arrow[r] 
  & B\PGL_2
\end{tikzcd}
\end{equation}
from which the conclusion follows immediately.
\end{proof}

Next, we show that for $g$ even the rigidification morphism $r_g^n: \wRH_g^n \to \RH_g^n$ is a $\mu_2$-root-gerbe. We give three proofs.

\begin{proof}[Proof $1$ of part $(i)$ in Proposition \ref{prop: root stacks}]

This proof uses the explicit presentations of the stacks in question. 
Assume first that $n$ is even. Then the morphism $r_g^n$ between the quotient stacks in Part $(i)$ of Theorem \ref{thm: g even presentation tilde} and that in Theorem \ref{thm: pres RH} is obtained as the base change of 
\[
B(\Gm \times \GL_2) \to B(\Gm \times \GL_2), \qquad 
(\lambda,A) \mapsto (\lambda^2, \lambda A),
\]
which factors as the composition of the isomorphism 
$(\lambda,A) \mapsto (\lambda, \lambda A)$ and the root gerbe 
$(\lambda,A) \mapsto (\lambda^2,A)$. In particular, $\wRH_g^n$ is obtained from $\RH_g^n$ by adding a root of $t$.

When $n$ is odd, the proof is analogous, except that $r_g^n$ is obtained as the base change of  
\[
B(\Gm \times \GL_2) \to B(\Gm \times \GL_2), \qquad 
(\lambda,A) \mapsto (\det(A), \lambda A),
\]
which factors as the composition of the isomorphism 
$(\lambda,A) \mapsto (\lambda^{-1}, \lambda A)$, the root gerbe 
$(\lambda,A) \mapsto (\lambda^2,A)$, and the isomorphism 
$(\lambda,A) \mapsto (\lambda \det(A),A)$.  In particular, in this case $\wRH_g^n$ is obtained from $\RH_g^n$ by adding a root of $t+c_1$.
\end{proof}

The next proof does not use the explicit presentations of the various stacks and it is conceptually more satisfying.

\begin{proof}[Proof $2$ of part $(i)$ in Proposition \ref{prop: root stacks}]

Consider the universal diagram over $\RH_g^n$

\begin{equation}\label{eq: universal diag}
\begin{tikzcd}[row sep=3em, column sep=4em]
\mathcal{C} \arrow[rr, "f"] \arrow[dr, "\pi"'] & & \cP\arrow[dl, "\rho"] \\
& \RH_g^n &
\end{tikzcd}
\end{equation}
where $\pi$ is the universal curve, $\rho$ the universal Brauer-Severi of relative dimension $1$ and $f$ the double cover. Let
\[
D_{2n}, \, D_{2g+2-2n} \subseteq \cP 
\qquad \text{and} \qquad 
e_{2n}, \, e_{2g+2-2n} \subseteq \mathcal{C}
\]
denote the two universal divisors of relative degree $2n$ and $2g+2-2n$, respectively. The divisors $e_{2n}$ and $e_{2g+2-2n}$ lie in the universal Weierstrass divisor $\mathcal{W} \subseteq \mathcal{C}$.

As explained in \cite[Section 5.1]{EF09}, the Brauer--Severi scheme $\cP$ is actually a projective bundle when $g$ is even, and more precisely $\cP = \P(V_g^\vee)$  where $V_g= \pi_*( \omega_{\pi}^{\otimes g/2}( (1-g/2)\mathcal{W})$. In particular, it carries an $\cO(1)$.

The line bundle on $\mathcal{C}$
\[
\widetilde{\eta} \;=\; f^*\cO(n)(-e_{2n})
\]
has relative degree $0$, and the same holds for its square. More is true: 
\[
\widetilde{\eta}^{\otimes 2} \;\cong\; f^*(\cO(2n)(-D_{2n}))  \cong \pi^* N^\vee
\]
 for some line bundle $N$ on $\RH_g^n$ such that $\rho^*N^\vee  \cong \cO(2n)(-D_{2n})$. In particular, adjoining a square root $\sqrt{N}$ of $N$ produces a $\mu_2$-gerbe over $\RH_g^n$ whose universal curve carries a universal line bundle 
\[
\eta \;=\; \widetilde{\eta}\otimes \sqrt{N}
\]
together with a universal isomorphism $\eta^{\otimes 2} \xrightarrow{\sim} \cO$. This $\mu_2$-gerbe over $\RH_g^n$ is precisely $\wRH_g^n$. 

In order to understand the class of $N$ in $\mathrm{Pic}(\RH_g^n)$, observe that
\begin{align*}
N &= \rho_* \rho^* N \\
  &= \rho_* \big( \cO(-2n)(D_{2n}) \big) \\
  &= \rho_* \big( (\cO(-2n) \otimes \omega_{\rho}^{\otimes -n}) \otimes (\omega_{\rho}^{\otimes n}(D_{2n})) \big) \\
  &= \big(\rho_* (\cO(-2) \otimes \omega_{\rho}^\vee) \big)^{\otimes n} \otimes \rho_* (\omega_{\rho}^{\otimes n}(D_{2n})),
\end{align*}
where in the last equality we have used \cite[Lemma 2.6]{Lan23}. 

From the Euler sequence 
$$
0 \to \omega_\rho \to \cO_{\cP}(-1) \otimes V_g \to \cO_{\cP} \to 0
$$
we obtain $\omega_\rho^\vee= \cO_{\cP}(2) \otimes \det(V_g)^\vee$. Since $\rho_* {\cO_{\cP}}= \cO_{\RH_g^n}$, it follows that

$$
\rho_*(\cO(-2) \otimes \omega_\rho^\vee) = \det(V_g)^\vee
$$
whose first Chern class is $c_1$ by \cite[Proposition 5.2]{EF09} \footnote{The sign is different from that in~\cite[Proposition 5.2]{EF09}. The reason is the same as the one in footnote $1$} On the other hand, the sheaf $\omega_\rho^{\otimes n}(D_{2n})$ is pulled back from the tautological Brauer–Severi variety of relative dimension $1$ over $\cD_{2n}$. Consequently, $\rho_*(\omega_\rho^{\otimes n}(D_{2n}))$ is also pulled back from $\cD_{2n}$, where it is identified with the line bundle $\mathcal{N}_n$ described in \S\ref{sec: interpretation generators g even} or \cite[Section~5.3]{CLI}. It follows from diagram~\eqref{diagram: square of groups} and~\cite[Section 5.3]{CLI} that
$
c_1(\rho_*(\omega_\rho^{\otimes n}(D_{2n}))) = -t,
$
and thus 
\[
c_1(N) = n c_1 - t,
\] 
which modulo $2$ is congruent to $t$ if $n$ is even, and to $c_1+t$ if $n$ is odd.
\end{proof}

\begin{proof}[Proof $3$ of part $(i)$ in Proposition \ref{prop: root stacks}]
Let $\widetilde{\mathcal{J}}^0_{g} \to \cM_g$ be the universal
Picard stack parametrizing degree $0$ line bundles on genus $g$ curves. 
This is a $\Gm$-gerbe over its rigidification $\mathcal{J}^0_g$. 
Denote by $\widetilde{\mathcal{JH}}^0_{g}$ and $\mathcal{JH}^0_{g}$ their restrictions to the hyperelliptic locus. 
By \cite[Corollary~1.5]{LarsonPic}, the Brauer class of the $\Gm$-gerbe 
$\widetilde{\mathcal{JH}}^0_{g} \to \mathcal{JH}^0_{g}$ is trivial when $g$ is even. In particular, we obtain a commutative diagram
$$
\begin{tikzcd}[row sep=3.5em, column sep=4em]
\wRH_g \arrow[r] \arrow[rd] 
  & \widetilde{\mathcal{JH}}^0_{g}[2] \smallsetminus \{ \cO\} = \RH_g \times B \Gm \arrow[r] \arrow[d] 
  & \widetilde{\mathcal{JH}}^0_{g}= \mathcal{JH}^0_{g} \times B \Gm \arrow[d] \\
& \RH_g \arrow[r] 
  & \mathcal{JH}^0_{g}
\end{tikzcd}
$$
where the right square is cartesian and $\widetilde{\mathcal{JH}}^0_{g}[2] \smallsetminus \{ \cO\}$ is parametrizing non-trivial $2$-torsion line bundles. By Lemma~\ref{lem: main new lemma gerbes}, the class $[\wRH_g] \in H^2_{\text{ét}}(\RH_g,\mu_2)$ maps to zero in $H^2_{\text{ét}}(\RH_g,\Gm)$ and the conclusion follows from the exact sequence in~\S\ref{sec: kummer sequence}.
\end{proof}

\subsubsection{Interpretation of the generators of $\CH^*(\wRH_g^n)$ for $g$ even}\label{sec: interpretation generators tilde g even}

By Corollary~\ref{cor: chow wRHg for g even}, it is enough to give a geometric interpretation to the class $u$, as the other classes are pulled back from $\RH_g^n$, which is already treated in~\S\ref{sec: interpretation generators g even}.

Using the notation of proof (2) of part (i) of Proposition~\ref{prop: root stacks}, we identified $\wRH_g^n$ with the root stack over $\RH_g^n$ of the line bundle $N=\rho_*(\cO(-2n)(D_{2n}))$, whose Chern class is $nc_1-t$. Here, $\cO(1)=\cO_{\P(V_g^{\vee})}(1)$. Moreover, we showed that
\[
    \sqrt{N}=\pi_*(\eta\otimes f^*\cO(-n)(e_{2n})),
\]
where $\eta$ is the universal 2-torsion line bundle on the universal hyperelliptic curve over $\wRH_g^n$. Therefore, when $n$ is even we have 
\[
    u=c_1\left(\pi_*\left(\eta\otimes f^*\cO(-n)(e_{2n}))\right)^\vee \otimes(\det V_g)^{ \otimes \frac{n}{2} }\right),
\]
while if $n$ is odd we have
\[
    u=c_1\left(\pi_*\left(\eta\otimes f^*\cO(-n)(e_{2n}))\right)^\vee \otimes(\det V_g)^{ \otimes \frac{n+1}{2}}\right).
\]

\subsection{Case $g$ odd}\label{sec: case g odd}

In this section we prove Theorem~\ref{thm: pres wRHg for g odd}, Theorem~\ref{thm: pres g+1/2}, and conclude the proof of Proposition~\ref{prop: root stacks}, starting with the presentations. From them, we derive Theorem~\ref{thm: Chow wRHgn odd g and n}. 

\begin{proof}[Proof of Theorem~\ref{thm: pres wRHg for g odd}]
    The proof is similar to for Theorem~\ref{thm: pres RH} in~\S\ref{sec: pres for g even}. By Proposition \ref{prop: starting point}, we know that $\wRH_g^n$ is an open in $\cH_{n-1}\times_{B\PGL_2}\cH_{g-n}$. Applying~\cite[Corollary 4.7]{AV04} we get diagrams similar to~\eqref{eq: diag pres wRh} and~\eqref{eq: big diag pres wRH}. When $n$ is even, the bottom right part of the analogue of diagram~\eqref{eq: big diag pres wRH} is
    \begin{equation}\label{eqn: diagram of groups for g odd}
    \begin{tikzcd}
        B(\Gm\times\Gm\times\PGL_2)\arrow[r,"\mathrm{pr}_{1,3}"]\arrow[d,"\mathrm{pr}_{2,3}"] & B(\Gm\times\PGL_2)\arrow[d,"\mathrm{pr}_2"]\\
        B(\Gm\times\PGL_2)\arrow[r,"\mathrm{pr}_2"] & B\PGL_2
    \end{tikzcd}
    \end{equation}
    where $\mathrm{pr}$ denotes the projection to the factor indicated in the subscript.
    Instead, when $n$ is odd, we have
    \[
    \begin{tikzcd}
        B(\Gm\times\GL_2)\arrow[r,"\rho"]\arrow[d,"\mathrm{pr}_1"] & B\GL_2\arrow[d]\\
        B\GL_2\arrow[r] & B\PGL_2
    \end{tikzcd}
    \]
    where $\rho(\lambda,A)=\lambda A$ and the maps $B \GL_2 \to B \PGL_2$ are the standard ones. The statement follows. 
\end{proof}

Next we derive Theorem~\ref{thm: Chow wRHgn odd g and n} from the above presentation.

\begin{proof}[Proof of Theorem~\ref{thm: Chow wRHgn odd g and n}]
    By Theorem~\ref{thm: pres wRHg for g odd}, $\wRH_g^n$ is the composite of two $\Gm$-torsors over the quotient stack $[\P^{2n}(\Sym^{2n}(V^{\vee}))\times\P^{2g+2-2n}(V^{\vee})/\Gm\times \GL_2]$, relative to $-\xi_{2n}+(n-1)c_1$ and $-\xi_{2g+2-2n}-2t+(g-n)c_1$, respectively. The statement follows from applying the $\Gm$-bundle formula twice, starting from the Chow ring in Theorem \ref{thm: chow P x P}, together with some simplifications using the first relation of the statement.
\end{proof}

\begin{remark}\label{rmk: Vg-n on wRHgn}
In the proof of Theorem~\ref{thm: pres wRHg for g odd}, where both $n$ and $g$ are odd, the map 
\[
\rho: B(\Gm \times \GL_2) \to B\GL_2
\] 
corresponds to the morphism $\wRH_g^n \to \cH_{g-n}$. In particular, if we denote by $V_{g-n}$ the rank-$2$ vector bundle on $\cH_{g-n}$ introduced in \S\ref{lem: interpretation generators g even}, then its pullback to $\wRH_g^n$ has first Chern class $-c_1-2t$ and second Chern class $c_2-tc_1+t^2$.
\end{remark}

\begin{proof}[Proof of Theorem~\ref{thm: pres g+1/2}]
    Recall that we have a $\mu_2$-torsor
    $$
    (\cH_{(g-1)/2}\times_{B\PGL_2}\cH_{(g-1)/2})\setminus\Delta\rightarrow\wRH_g^{(g+1)/2}
    $$
    which is a $\mu_2$-torsor under the action exchanging the two factors. Suppose that $n=(g+1)/2$ is even. Recall that $G=(\Gm\times\Gm)\rtimes\mu_2$. Since the restriction of $A\otimes W_{(g+1)/2}$ to $\Gm\times\Gm\times\PGL_2$ is isomorphic to $(\chi^{(1)})^{-2}\otimes W_{\frac{g+1}{2}}\times (\chi^{(2)})^{-2}\otimes W_{\frac{g+1}{2}}$, there is a morphism
    \[
    \begin{tikzcd}
        \pi:\left[\frac{(\chi^{(1)})^{-2} \otimes W_{\frac{g+1}{2}} \times (\chi^{(2)})^{-2} \otimes W_{\frac{g+1}{2}} \smallsetminus \Delta}{\Gm^{(1)}\times\Gm^{(2)} \times\PGL_2}\right]\arrow[r] & \left[\frac{A\otimes W_{\frac{g+1}{2}}}{G\times\PGL_2}\right]
    \end{tikzcd}
    \]
   The domain of $\pi$ admits a $\mu_2$-action given by swapping the two factors and the two copies of $\Gm$. The source of $\pi$ is isomorphic to $(\cH_{(g-1)/2}\times_{B\PGL_2}\cH_{(g-1)/2})\setminus\Delta$ and the two actions of $\mu_2$ are identified under this isomorphism, hence the target of $\pi$ is isomorphic to $\wRH_g^{(g+1)/2}$, as wanted. The case when $n$ is odd is completely analogous.
\end{proof}

Now, we proceed on proving Proposition~\ref{prop: root stacks}.
First, we deal with the case $n \in \{1,\ldots,(g-1)/2\}$ is even, proving that in this case the morphism $r_g^n$ is a root gerbe. As for the case $g$ even, we present two proofs: one that uses the explicit presentation of $\wRH_g^n$ as a quotient stacks and one which is more geometric.

\begin{proof}[Proof 1 of part (ii) of Proposition~\ref{prop: root stacks}]
    The morphism $r_g^n$ between the presentation of $\wRH_g^n$ as the quotient stack from part (ii) of Theorem~\ref{thm: pres wRHg for g odd} and the presentation of $\RH_g^n$ in~\cite[Theorem~1.19]{CLI}  is obtained by base change from the morphism
    \[
    \begin{tikzcd}
        B(\Gm\times\Gm\times\PGL_2)\arrow[r] & B(\Gm\times\Gm\times\PGL_2), & (\lambda_1,\lambda_2,[A])\mapsto(\lambda_1^2,\lambda_1^{-1}\lambda_2,[A]).
    \end{tikzcd}
    \]
    Therefore, it is a $\mu_2$-root gerbe relative to the first Chern class $x_1$ of the standard character of the first copy of $\Gm$. With the notation of~\cite{CLI}, this corresponds to $-\xi_{2n}$, and the statement follows.
\end{proof}

\begin{proof}[Proof 2 of part (ii) of Proposition~\ref{prop: root stacks}]
    Consider the same universal diagram as in~\eqref{eq: universal diag}, and use the same notation for the universal divisors of $\cP$. Recall that $n$ is even, and consider the line bundle
    \[
        \widetilde{\eta}:=f^*(\omega_\rho^{\otimes n/2})\otimes\cO_{\mathcal{C}}(e_{2n}).
    \]
    Then, $\widetilde{\eta}^{\otimes2}= f^*(\omega_{\rho}^{\otimes n}\otimes\cO_{\cP}(D_{2n}))$ is trivial on the fibers of $\pi$, hence $\widetilde{\eta}^{\otimes 2}= \pi^*M^{\vee}$, for some a line bundle $M$ on $\RH_g^n$ such that $\rho^*M^{\vee}=\omega_{\rho}^{\otimes n}\otimes\cO_{\cP}(D_{2n})$. Given a root $\sqrt{M}$ of $M$, then we can form $\eta=\widetilde{\eta}\otimes\sqrt{M}$, obtaining the universal 2-torsion line bundle with an isomorphism $\eta^{\otimes 2} \xrightarrow{\sim} \cO$. This shows that $\sqrt{(\RH_g^n,M)}\cong\wRH_g^n$, and we are only left with computing $c_1(M)$.
    
    Let $\cL$ be the line bundle on $\cP$ such that $f_*\cO=\cO\oplus\cL$, which has degree $-g-1$ on the fibers of $\cP\rightarrow\RH_g^n$. Then,
    \begin{align*}
        M^{\vee}&=\pi_*f^*(\omega_{\rho}^{\otimes n}\otimes\cO(D_{2n}))\\
        &\cong \rho_*f_*f^*(\omega_{\rho}^{\otimes n}\otimes\cO(D_{2n}))\\
        &\cong\rho_*(\omega_{\rho}^{\otimes n}\otimes\cO(D_{2n})\otimes(\cO\oplus\cL))\\
        &\cong\rho_*(\omega_{\rho}^{\otimes n}\otimes\cO(D_{2n})),
    \end{align*}
    where we have used that $\omega_{\rho}^{\otimes n}\otimes\cO(D_{2n})\otimes\cL$ has negative degree on the fibers of $\rho$. The line bundle we obtained is called $\mathcal{N}_n$ in~\cite[Section 5.3]{CLI}, where it is also shown to have first Chern class $\xi_{2n}$. This concludes.
\end{proof}

\begin{remark}
    When $g$ is odd and $1\leq n\leq(g-1)/2$, the Chow ring of $\RH_g^n\rightarrow\cD_{2n,2g+2-2n}$ is the $\mu_2$-root gerbe given by adding a square root $t$ of $\xi_{2n}+\xi_{2g+2-2n}$, see~\cite[Lemma 1.15]{CLI}. It follows that in point (ii) of Proposition~\ref{prop: root stacks} we could have added a root of $\xi_{2g+2-2n}$ instead of $\xi_{2n}$. Finally, this shows that $\wRH_g^n\rightarrow\cD_{2n,2g+2-2n}$ is the a double root gerbe associated to $\xi_{2n}$ and $\xi_{2g+2-2n}$, as one could also see directly from the presentations.
\end{remark}

\begin{remark}\label{rmk: no third proof for g odd}
    By \cite[Corollary~1.5]{LarsonPic}, the Brauer class of the $\Gm$-gerbe 
    $\widetilde{\mathcal{JH}}^0_{g} \to \mathcal{JH}^0_{g}$ is non-trivial when $g$ is odd, 
    which prevents us from obtaining a third, distinct proof of Theorem~\ref{prop: root stacks}. 
    In this sense, the statement that when both $g$ and $n$ are odd the map $r_g^n$ is not a root gerbe can be viewed as a strengthening of that result.
\end{remark}

Now, we are ready to prove that cases (i) and (ii) in Proposition~\ref{prop: root stacks} are the only ones where we get a root stack.

\begin{proof}[Conclusion of the proof of Proposition~\ref{prop: root stacks}]
    We claim that whenever $g$ is odd and $n$ is odd or equal to $(g+1)/2$ the pullback $(r_g^n)^*:\Pic(\RH_g^n)\rightarrow\Pic(\wRH_g^n)$ is an isomorphism. The statement then follows from the claim and Lemma~\ref{lem: criterion root gerbe with Pic}.
    
    \begin{enumerate}
        \item[$\bullet$] Suppose that $n$ is odd and $n\not=(g+1)/2$.
        
        Then $\RH_g^n \to \cD_{2n,\,2g+2-2n}$ is a root gerbe by~\cite[Lemma~1.15]{CLI}, 
        and hence the pullback $\Pic(\cD_{2n,\,2g+2-2n}) \to \Pic(\RH_g^n)$ has index two. Since $\Pic(\RH_g^n) \to \Pic(\wRH_g^n)$ has either index~$2$ or is an isomorphism, it suffices to show that the image of the pullback $\rho : \Pic(\cD_{2n,\,2g+2-2n}) \longrightarrow \Pic(\wRH_g^n)$ does not have index~$4$. 
        
        By Theorem~\ref{thm: pres wRHg for g odd} and~\cite[Section 1.2.3]{CLI}, there is a commutative diagram
    \[
    \begin{tikzcd}
        \Pic(B(\Gm\times\Gm\times\PGL_2))\arrow[d,twoheadrightarrow]\arrow[r,"\phi^*"] & \Pic(B(\Gm\times\GL_2))\arrow[d,twoheadrightarrow]\\
        \Pic(\cD_{2n,2g+2-2n})\arrow[r] & \Pic(\wRH_g^n)
    \end{tikzcd}
    \]
     where $\phi^*$ is the pullback along $\phi:B(\Gm\times\GL_2)\rightarrow B(\Gm\times\Gm\times\PGL_2)$ induced by $(\lambda,A)\mapsto(\det A,\lambda^2\det A,[A])$, and the two vertical arrow are surjective. The conclusion then follows from the fact that the image of $\phi^*$ has index~$2$, and therefore the bottom map has index at most~$2$.

    \item[$\bullet$] Suppose that $n=(g+1)/2$ is odd 

    The base change of  $\wRH_g^n\rightarrow[\cD_{g+1,g+1}/\mu_2]$ along the projection $\cD_{g+1,g+1}\rightarrow[\cD_{g+1,g+1}/\mu_2]$ is just $(\cH_{n-1} \times_{B\PGL_2} \cH_{g-n}) \smallsetminus \Delta \to \cD_{g+1,g+1}$ which is not a root by the argument in the previous bullet.

    \item[$\bullet$] Suppose $n=(g+1)/2$ is even.

    Then, $\RH_g^{(g+1)/2}\rightarrow[\cD_{g+1,g+1}/\mu_2]$ is a root gerbe by~\cite[Lemma 1.1. and Section 1.2.1]{CLII}. Thus, it is enough to show that $\Pic([\cD_{g+1,g+1}/\mu_2]) \to \Pic(\wRH_g^{(g+1)/2})$ has index $<4$ to conclude. By Theorem~\ref{thm: pres g+1/2} and~\cite[Theorem 1.3]{CLII}, there is a commutative diagram with surjective vertical homomorphisms
    \[
    \begin{tikzcd}
        \Pic(B(G\times\PGL_2))\arrow[r,"\phi^*"]\arrow[d,twoheadrightarrow] & \Pic(B(G\times\PGL_2))\arrow[d,twoheadrightarrow]\\
        \Pic([\cD_{g+1,g+1}/\mu_2])\arrow[r] & \Pic(\wRH_g^{\frac{g+1}{2}})
    \end{tikzcd}
    \]
    where $\phi^*$ is the pullback along $B(G\times\PGL_2)\rightarrow B(G\times\PGL_2)$ induced by the endomorphism of $G$ in Notation~\ref{not: representation A of G}. Recall that $\Pic(BG)\cong\beta_1\Z\oplus\gamma\Z/2\Z$, see~\cite[Theorem 5.2]{Lar19}. Then, $\phi^*(\beta_1)=2\beta_1+\gamma$ and $\phi^*\gamma=\gamma$, thus $\phi^*$ has index 2, concluding this last case.
    \end{enumerate}
\end{proof}

\subsubsection{Interpretation of the generators of $\CH^*(\wRH_g^n)$ for odd $g$}\label{sec: interpretation of the generators tilde g odd}

In this section $g$ is always odd and $1 \leq n \leq (g-1)/2$. We distinguish two cases.

Suppose that $n$ is even. With the notation of Corollary~\ref{cor: chow wRHg for g odd n even}, we aim to provide natural vector bundles whose Chern classes are $c_2$, $c_3$, $t_{2n}$ and $t_{2g+2-2n}$, as $c_1\in I$. For the first two, denote with $\cE_{g-n}$ the rank 3 vector bundle on $\cH_{g-n}$ functorially defined as
\[
    \cE_{g-n}=\pi_*\omega_{\pi}^{\vee}(W)
\]
for every family of hyperelliptic curves $C\xrightarrow{\pi}S$ with Weierstrass divisor $W\subset C$. Note that $c_2$ and $c_3$ are pulled back from $B\PGL_2$, and by~\cite[Section~7]{DL18} we have
\[
    c_2 = c_2(\cE_{g-n}), \qquad c_3 = c_3(\cE_{g-n}),
\]
on $\cH_{g-n}$, and the same equalities remain true after pulling back to $\wRH_g^n$. Furthermore, when $n-1 \geq 2$, one also has $c_i = c_i(\cE_{n-1})$.

For the other two classes, recall that we are thinking of $\wRH_g^n$ as a double root gerbe over $\cD_{2n,2g+2-2n}\subset\cD_{2n}\times_{B\PGL_2}\cD_{2g+2-2n}$ relative to $\xi_{2n}$ and $\xi_{2g+2-2n}$. Now, $\cH_{n-1}\rightarrow\cD_{2g+2-2n}$ is the root gerbe relative to $\xi_{2n}$, and similarly for $\cH_{g-n}$, see~\cite[Section 5.3]{CLI}. Thus, $t_{2n}$ and $t_{2g+2-2n}$ are pulled back from $\cH_{n-1}$ and $\cH_{g-n}$, respectively. Recall the line bundles $\cL_{n-1}$ and $\cL_{g-n}$ introduced in~\S\ref{sec: interpretation generators tilde g even}. Then, by the discussion in~\cite[Section~5.3]{CLI} together with diagram~\eqref{eqn: diagram of groups for g odd}, we deduce
\[
   t_{2n} = c_1(\cL_{n-1}), \qquad t_{2g+2-2n} = c_1(\cL_{g-n}).
\]

This completes the discussion for the case where $n$ is even.

Now, assume $n$ to be odd. By Theorem~\ref{thm: Chow wRHgn odd g and n}, the generators are $t$, $c_1$ and $c_2$. Since $g-n$ is even, as in~\S\ref{sec: interpretation generators g even} we have a well-defined rank $2$ vector bundle $V_{g-n}$ and by Remark \ref{rmk: Vg-n on wRHgn}, we have
\begin{align*}
    c_1(V_{g-n})=-c_1-2t, && c_2(V_{g-n})=c_2+tc_1+t^2.
\end{align*}
Using the presentation of $\wRH_g^n$ in Theorem~\ref{thm: pres wRHg for g odd} and the presentation of $\cH_g$ in~\cite[Corollary 4.7]{AV04}, we have a commutative diagram
\[
\begin{tikzcd}
    \wRH_g^n\arrow[r]\arrow[d] & \cH_g\arrow[d]\\
    B(\Gm\times\GL_2)\arrow[r,"\phi"] & B(\Gm\times\PGL_2)
\end{tikzcd}
\]
where the vertical morphisms induce surjections on Chow rings by pullback, and $\phi$ is induced by $(\lambda,A)\mapsto(\lambda\cdot\det A,[A])$. In particular, if $x$ denotes the first Chern class of the standard character of the copy of $\Gm$ in the bottom right corner, we have $\phi^*(x)=c_1+t$. From~\cite[Section~7]{DL18}, it follows that $c_1(\cL_g)=c_1+t$ and thus
\[
    t = c_1\!\left(\cL_g^{\vee}\otimes \det(V_{g-n}^{\vee})\right), 
    \qquad 
    c_1 = c_1\!\left(\cL_g \otimes V_{g-n}\right), \qquad c_2= c_2( \cL_g \otimes V_{g-n}).
\]

\appendix

\appendixsection{Gerbes induced by homomorphisms of groups}\label{sec: preliminaries on gerbes}
In this appendix we give a criterion for determining the induced $G$-gerbe given an $H$-gerbe and an homomorphism $H\rightarrow G$ of abelian sheaves over an algebraic stack $\cY$.
Along the way, we also recall some known facts about abelian gerbes, in particular regarding root gerbes; we refer to~\cite[Section 3.4]{Gir65},~\cite[Appendix B]{AGV08}, and~\cite{Cad04} for the details. We will be mostly interested in the case where the groups involved are either $\mu_2$ or $\Gm$, so at times we will specialize the discussion to that situation.

\begin{setup}\label{setup: preliminaries gerbes}
    Throughout this appendix we work over a fixed Noetherian base scheme $S$ and denote with $\cY$ an algebraic stack over $S$. We denote by $H$ and $G$ two abelian group schemes that are flat, separated, and of finite presentation over $S$, which we will also interpret as abelian sheaves on $\cY$.
\end{setup}

Recall that, in the above setting, the second étale cohomology group $\mathrm{H}_{\text{fppf}}^2(\cY,G)$ is in bijective correspondence with classes of isomorphism of $G$-gerbes. This allows the use of cohomological methods when studying abelian gerbes.

\subsection{Remarks on the Kummer sequence}\label{sec: kummer sequence}

One application of the cohomological characterization of abelian gerbes is to the Kummer sequence; in this subsection, we assume $S$ to be the spectrum of a field of characteristic different from 2. Then, the Kummer sequence for $\mu_2$ reads
\[
\begin{tikzcd}
    0\arrow[r] & \mu_2\arrow[r,"i"] & \Gm\arrow[r] & \Gm\arrow[r] & 0,
\end{tikzcd}
\]
and induces the exact sequence
\begin{equation}\label{eqn: kummer}
\begin{tikzcd}
    \Pic(\cY)\cong\mathrm{H}_{\text{ét}}^1(\cY,\Gm)\arrow[r] & \mathrm{H}_{\text{ét}}^2(\cY,\mu_2)\arrow[r,"i_*"] & \mathrm{H}_{\text{ét}}^2(\cY,\Gm).
\end{tikzcd}
\end{equation}
In this situation, the connecting homomorphism $\mathrm{Pic}(\cY) \to \mathrm{H}_{\text{ét}}^2(\cY,\mu_2)$ admits a particularly explicit description: it sends a line bundle $L$ to the root gerbe $\sqrt{(\cY,L)}$.
Here, we used the isomorphisms
\[
\mathrm{H}_{\mathrm{fppf}}^2(\cY,G)\cong \mathrm{H}_{\acute{e}t}^2(\cY,G)
\]
for $G=\mu_2$ or $G=\Gm$, the first of which holds only under the assumption that the characteristic is different from $2$. In particular, a $\mu_2$-gerbe $\cX\rightarrow\cY$ is a root gerbe if and only if the associated class in $\mathrm{H}_{\text{ét}}^2(\cY,\mu_2)$ has trivial image under the homomorphism $\mathrm{H}_{\text{ét}}^2(\cY,\mu_2)\rightarrow\mathrm{H}_{\text{ét}}^2(\cY,\Gm)$ induced by the canonical immersion $\mu_2\hookrightarrow\Gm$.

The following is a useful criterion for understanding if a $\mu_2$-gerbe over $\cY$ is a root gerbe, which we will exploit to prove part of Proposition~\ref{prop: root stacks}. We recall the proof for the convenience of the reader.

\begin{lemma}\label{lem: criterion root gerbe with Pic}
    Let $f:\cX\rightarrow\cY$ be a $\mu_2$-gerbe between connected algebraic stacks locally of finite presentation over $S$. Then, $f$ is not a root gerbe if and only if $f^*:\Pic(\cY)\rightarrow\Pic(\cX)$ is an isomorphism.
\end{lemma}
\begin{proof}
    By~\cite[Theorem 1.2, Remark 1.4]{Lop23}, given such a $\mu_2$-gerbe $h:\cX\rightarrow\cY$, there exists an exact sequence
    \[
    \begin{tikzcd}
        0\arrow[r] & \Pic(\cY)\arrow[r,"f^*"] & \Pic(\cX)\arrow[r] & \Z/2\Z\arrow[r,"\phi"] & \mathrm{H}_{\text{ét}}^2(\cY,\Gm),
    \end{tikzcd}
    \]
    where $\phi$ sends $-1$ to the image of $[\cX\rightarrow\cY]\in\mathrm{H}_{\text{ét}}^2(\cY,\mu_2)$ along the homomorphism $\mathrm{H}_{\text{ét}}^2(\cY,\mu_2)\rightarrow\mathrm{H}_{\text{ét}}^2(\cY,\Gm)$. It follows that $\cX\rightarrow\cY$ is not a root gerbe if and only if $\phi$ is not the zero map, which is equivalent to $f^*$ being an isomorphism.
\end{proof}

\subsection{Characterizations of gerbes induced by homomorphisms of groups}

Let $\phi:H\rightarrow G$ be an homomorphism of abelian group schemes over $S$. Then, for every $H$-gerbe $h:\cX\rightarrow\cY$ there is an induced $G$-gerbe
\[
\begin{tikzcd}
    g:(BG\times_{S}\cX)\rigid\!\! H\arrow[r] & \cY
\end{tikzcd}
\]
where the $H$-rigidification is with respect to the $H$-2-structure on $BG\times_{S}\cX$ where $H$ acts on the automorphism of $\cX$ via the $H$-gerbe structure on $\cX$, and on $BG$ via $u\cdot v=v-\phi(u)$ for $u\in H$, $v\in G$. We refer to~\cite[Section 5]{Rom05},~\cite[Section 5]{ACV03},~\cite[Appendix A]{AOV08} and~\cite[Appendix C]{AGV08} for details on the rigidification operation.

Moreover, the composite of $\cX\rightarrow BG\times_{S}\cX$ and the rigidification map $ \cX\rightarrow BG\times_{S}\cX \to (BG\times_{S}\cX)\rigid\!\! H$ defines a canonical morphism
\[
\begin{tikzcd}
    \Phi:\cX\arrow[r] & (BG\times_{S}\cX)\rigid\!\!H
\end{tikzcd}
\]
such that $h=g\circ\Phi$, and it is $H$-2-equivariant with respect to the action of $H$ on $\cX$, in the sense of~\cite[Appendix C.2]{AGV08}. We want to show that this map is unique, in the following sense.

Every $G$-2-structure on a stack $\cZ$ with $G$-rigidification $\cY$ has an associated $G$-2-action map
\[
\begin{tikzcd}
    \rho:BG\times_{S}\cZ\arrow[r] & \cZ
\end{tikzcd}
\]
such that $BG\times_{S}\cZ\xrightarrow{\cong}\cZ\times_{\cY}\cZ$ is an isomorphism, see~\cite[Appendix C.3]{AGV08}. For us, the action of $G$ on the product $P_1 \times_S P_2$ of two $G$-torsors $P_1,P_2$ is given by 
$$
v \cdot (p_1,p_2)= (v^{-1} \cdot p_1, v \cdot p_2) \ \text{for} \ v \in G \ \text{and} \ p_1 \in P_1, p_2 \in P_2.
$$
Then, the quotient $(P_1 \times_S P_2)/G$ admits a natural (left) action of $G$ given by that on $P_2$.

Now, suppose given a commutative diagram
\begin{equation}\label{eq: commutative diag gerbes}
\begin{tikzcd}
    \cX\arrow[rd,"h"']\arrow[rr,"f"] && \cZ\arrow[ld,"g"]\\
    & \cY
\end{tikzcd}
\end{equation}
where $h$ and $g$ are $H$ and $G$-gerbes respectively, while $f$ is $H$-2-equivariant with respect to $\phi$, in the sense of~\cite[Appendix C.2]{AGV08}.
Consider the morphism $\psi:BG\times_{S}\cX\rightarrow\cZ$ defined as the composite
\[
\begin{tikzcd}
    BG\times_{S}\cX\arrow[r,"1\times f"] & BG\times_{S}\cZ\arrow[r,"\rho"] & \cZ,
\end{tikzcd}
\]
where $\rho$ is the $BG$-action map on $\cZ$. The next lemma shows that $\psi$ identifies $\cZ$ with the $H$-rigidification of $\cX\times_S BG$, and $f$ with the canonical morphism $\Phi$.

\begin{lemma}\label{lem: isom gerbes}
    With the notation above, the morphism $\psi$ induces an isomorphism
    \[
    \begin{tikzcd}
        \overline{\psi}:(BG\times_{S}\cX)\rigid\!\! H\arrow[r] & \cZ,
    \end{tikzcd}
    \]
    such that $\overline{\psi}\circ\Phi=f$.
\end{lemma}
\begin{proof}
    By~\cite[Theorem 5.1.5]{ACV03}, the morphism $\psi$ passes to the rigidification if and only if for every scheme $T$ and object $\xi\in(BG\times_{S}\cX)(T)$ the composite $H(T)\hookrightarrow\Aut_T(\xi)\rightarrow\Aut_T(\psi(\xi))$ is trivial. This can be checked smooth-locally on $\cY$, as well as for the rest of the statement. In particular, we may assume that $\cY$ is a scheme and both $h$ and $g$ are trivial gerbes, in particular $\cX\cong\cY\times_{S}BH$ and $\cZ\cong\cY\times_{S}BG$.

    Then, $\psi$ is the composite of
    \[
    \begin{tikzcd}
        \cY\times_{S}BG\times_{S}BH\arrow[r,"1\times f"] & \cY\times_{S}BG\times_{S}BG\arrow[r,"\rho"] & \cY\times_{S}BG
    \end{tikzcd}
    \]
    where $\rho$ is the multiplication map. Since $f$ is $H$-2-equivariant with respect to $\phi$, for every $\cY$-scheme $T$ and object $\xi\in(\cY\times_{S}BH\times_{S}BG)(T)$ the induced morphism $\Aut_T(\xi)\rightarrow\Aut_T(\psi(\xi))$ is the composite of
    \begin{equation}\label{eq: psi on automorphisms}
    \begin{tikzcd}[row sep=tiny, column sep=small]
        \Aut_T(\xi)\cong G(T)\times H(T)\arrow[r] & G(T)\times G(T)\arrow[r] & G(T)\cong\Aut_T(\psi(\xi)),\\
        (v,u)\arrow[r,mapsto] & (v,\phi(u))\arrow[r,mapsto] & v+\phi(u).
    \end{tikzcd}
    \end{equation}
     In particular, the composite of $H(T)\rightarrow G(T)\times H(T)$, $u\mapsto(-\phi(u),u)$ with the above map is trivial. By~\cite[Theorem 5.1.5]{ACV03}, we get the desired morphism $\overline{\psi}$.
    
    Moreover, the composite of the inclusion $G(T)\hookrightarrow G(T)\times H(T)$, $v\mapsto (v,0)$ with the map in~\eqref{eq: psi on automorphisms} is the identity, showing that $\overline{\psi}$ is $G$-2-equivariant. Then, one concludes by noticing that a $G$-2-equivariant morphism of $G$-gerbes over $\cY$ is always an isomorphism.
\end{proof}

There is another natural way to obtain a $G$-gerbe from an $H$-gerbe 
$\cX \to \cY$ and a morphism $\phi \colon H \to G$: push forward the class 
$[\cX \to \cY] \in \mathrm{H}^2_{\mathrm{fppf}}(\cY,H)$ along the induced map 
\[
\phi_* \colon \mathrm{H}^2_{\mathrm{fppf}}(\cY,H) \longrightarrow 
\mathrm{H}^2_{\mathrm{fppf}}(\cY,G),
\]
and then take the associated $G$-gerbe. 
The following lemma shows that these two constructions agree.

\begin{lemma}\label{lem: main new lemma gerbes}
    With the notation above, we have
    \[
        \phi_*([\cX\rightarrow\cY])=[(BG\times_{S}\cX)\rigid\!\! H\rightarrow\cY].
    \]
    In particular, given a commutative diagram as in~\eqref{eq: commutative diag gerbes}, we have
    \[
        \phi_*([\cX\rightarrow\cY])=[\cZ\rightarrow\cY].
    \]
\end{lemma}
\begin{proof}
Let $0\rightarrow H\rightarrow I^{\bullet}$ and $0\rightarrow G\rightarrow J^{\bullet}$ be injective resolutions as abelian sheaves on the fppf site of $\cY$. Given a class $\alpha\in\mathrm{H}_{\text{fppf}}^2(\cY,H)$ and a representative $\tau\in I^2$, the associated $H$-gerbe $\cX_{\tau}$ is the stack whose objects are lifts $\gamma$ of $\tau$ to $I^1$, and an isomorphism between two objects $\gamma_1$ and $\gamma_2$ is a section $\theta$ of $I^0$ whose boundary is $\gamma_1-\gamma_2$. In particular, $H$ injects naturally in the automorphism groups of objects of $\cX_{\alpha}$ via $G\hookrightarrow I^0$. The isomorphism class of $\cX_{\tau}$ as an $H$-gerbe does not depend on the choice of $\tau$. The same holds for $G$. See~\cite[Exercise 6.4.44]{Alp25}.

By the standard properties of injective resolutions, the morphism $\phi$ extends to a morphism of complexes
\[
\begin{tikzcd}
    0\arrow[r] & H\arrow[r]\arrow[d,"\phi"] & I^0\arrow[r]\arrow[d,"\phi^0"] & I^1\arrow[r]\arrow[d,"\phi^1"] & I^2\arrow[r]\arrow[d,"\phi^2"] & \ldots\\
    0\arrow[r] & G\arrow[r] & J^0\arrow[r] & J^1\arrow[r] & J^2\arrow[r] & \ldots
\end{tikzcd}
\]
uniquely up to homotopy. Let $\beta=\phi_*(\alpha)$, and $\rho=\phi^2(\tau)$, thus $[\rho]=\beta$. Then, there is a natural morphism $\cX_{\tau}\rightarrow\cX_{\rho}$ sending an object $\gamma\in I^1$ to $\phi^1(\gamma)$ and a morphism $\theta\in I^0$ to $\phi^0(\theta)$. By the commutativity of the above diagram, this is a $H$-2-equivariant morphism with respect to $\phi$. By Lemma~\ref{lem: isom gerbes}, it follows that $\cX_{\tau}\cong(BG\times_{S}\cX)\rigid\!\! H$, as wanted.
\end{proof}

\bibliographystyle{amsalpha}
\bibliography{library}

$\,$\
\noindent

$\,$\
\noindent
\textsc{Department of Pure Mathematics {\it \&} Mathematical Statistics, 
University of Cambridge, Cambridge, UK}

\textit{e-mail address:} \href{mailto:au270@cam.ac.uk}{ac2758@cam.ac.uk}

$\,$\
\noindent

$\,$\
\noindent
\textsc{Department of Pure Mathematics, Brown University, 151 Thayer Street, Providence, RI 02912, USA}

\textit{e-mail address:} \href{mailto:alberto_landi@brown.edu}{alberto\_landi@brown.edu}

\end{document}